\documentclass[10pt, reqno]{amsart}
\usepackage{appendix}
\usepackage[a4paper, centering]{geometry}
\usepackage{xcolor}

\usepackage[colorlinks=true,hyperindex,pagebackref,linktocpage=true]{hyperref}
\usepackage{cleveref}
\usepackage{amsfonts}
\usepackage{amssymb}
\usepackage{tikz-cd}

\hypersetup{
	colorlinks,
	linkcolor={red!50!black},
	citecolor={blue!50!black},
	urlcolor={blue!80!black}
}

\usepackage{color}

\usepackage{stmaryrd}
\usepackage{mathrsfs}  
\usepackage{varwidth}
\theoremstyle{plain}
\newtheorem{theorem}{Theorem}[section]
\newtheorem{proposition}[theorem]{Proposition}
\newtheorem{lemma}[theorem]{Lemma}
\newtheorem{corollary}[theorem]{Corollary}
\newtheorem{conjecture}[theorem]{Conjecture}
\theoremstyle{definition}
\newtheorem{definition}[theorem]{Definition}

\newtheorem{remark}[theorem]{Remark}
\newtheorem{hypothesis}[theorem]{Hypothesis}

\newcommand{\nc}{\newcommand}
\nc{\on}{\operatorname}

\nc{\Q}{\mathbb{Q}}
\nc{\Z}{\mathbb{Z}}
\nc{\cl}{\mathrm{cl}}

\nc{\fraka}{{\mathfrak a}} \nc{\bba}{{\mathbf a}}
\nc{\frakb}{{\mathfrak b}}
\nc{\frakc}{{\mathfrak c}}
\nc{\frakd}{{\mathfrak d}}
\nc{\frake}{{\mathfrak e}}
\nc{\frakf}{{\mathfrak f}}
\nc{\frakg}{{\mathfrak g}}
\nc{\frakh}{{\mathfrak h}}
\nc{\fraki}{{\mathfrak i}}
\nc{\frakj}{{\mathfrak j}}
\nc{\frakk}{{\mathfrak k}}
\nc{\frakl}{{\mathfrak l}}
\nc{\frakm}{{\mathfrak m}}
\nc{\frakn}{{\mathfrak n}}
\nc{\frako}{{\mathfrak o}}
\nc{\frakp}{{\mathfrak p}}
\nc{\frakq}{{\mathfrak q}}
\nc{\frakr}{{\mathfrak r}}
\nc{\fraks}{{\mathfrak s}}
\nc{\frakt}{{\mathfrak t}}
\nc{\fraku}{{\mathfrak u}}
\nc{\frakv}{{\mathfrak v}}
\nc{\frakw}{{\mathfrak w}}
\nc{\frakx}{{\mathfrak x}}
\nc{\fraky}{{\mathfrak y}}
\nc{\frakz}{{\mathfrak z}}
\nc{\frakA}{{\mathfrak A}}
\nc{\frakB}{{\mathfrak B}}
\nc{\frakC}{{\mathfrak C}}
\nc{\frakD}{{\mathfrak D}}
\nc{\frakE}{{\mathfrak E}}
\nc{\frakF}{{\mathfrak F}}
\nc{\frakG}{{\mathfrak G}}
\nc{\frakH}{{\mathfrak H}}
\nc{\frakI}{{\mathfrak I}}
\nc{\frakJ}{{\mathfrak J}}
\nc{\frakK}{{\mathfrak K}}
\nc{\frakL}{{\mathfrak L}}
\nc{\frakM}{{\mathfrak M}}
\nc{\frakN}{{\mathfrak N}}
\nc{\frakO}{{\mathfrak O}}
\nc{\frakP}{{\mathfrak P}}
\nc{\frakQ}{{\mathfrak Q}}
\nc{\frakR}{{\mathfrak R}}
\nc{\frakS}{{\mathfrak S}}
\nc{\frakT}{{\mathfrak T}}
\nc{\frakU}{{\mathfrak U}}
\nc{\frakV}{{\mathfrak V}}
\nc{\frakW}{{\mathfrak W}}
\nc{\frakX}{{\mathfrak X}}
\nc{\frakY}{{\mathfrak Y}}
\nc{\frakZ}{{\mathfrak Z}}
\nc{\bbA}{{\mathbb A}}
\nc{\bbB}{{\mathbb B}}
\nc{\bbC}{{\mathbb C}}
\nc{\bbD}{{\mathbb D}}
\nc{\bbE}{{\mathbb E}}
\nc{\bbF}{{\mathbb F}} \nc{\bbf}{{\mathbf f}}
\nc{\bbG}{{\mathbb G}}
\nc{\bbH}{{\mathbb H}}
\nc{\bbI}{{\mathbb I}}
\nc{\bbJ}{{\mathbb J}}
\nc{\bbK}{{\mathbb K}}
\nc{\bbL}{{\mathbb L}}
\nc{\bbM}{{\mathbb M}}
\nc{\bbN}{{\mathbb N}}
\nc{\bbO}{{\mathbb O}}
\nc{\bbP}{{\mathbb P}}
\nc{\bbQ}{{\mathbb Q}}
\nc{\bbR}{{\mathbb R}}
\nc{\bbS}{{\mathbb S}}
\nc{\bbT}{{\mathbb T}}
\nc{\bbU}{{\mathbb U}}
\nc{\bbV}{{\mathbb V}}
\nc{\bbW}{{\mathbb W}}
\nc{\bbX}{{\mathbb X}}
\nc{\bbY}{{\mathbb Y}}
\nc{\bbZ}{{\mathbb Z}}
\nc{\calA}{{\mathcal A}}
\nc{\calB}{{\mathcal B}}
\nc{\calC}{{\mathcal C}}
\nc{\calD}{{\mathcal D}}
\nc{\calE}{{\mathcal E}}
\nc{\calF}{{\mathcal F}}
\nc{\calG}{{\mathcal G}}
\nc{\calH}{{\mathcal H}}
\nc{\calI}{{\mathcal I}}
\nc{\calJ}{{\mathcal J}}
\nc{\calK}{{\mathcal K}}
\nc{\calL}{{\mathcal L}}
\nc{\calM}{{\mathcal M}}
\nc{\calN}{{\mathcal N}}
\nc{\calO}{{\mathcal O}}
\nc{\calP}{{\mathcal P}}
\nc{\calQ}{{\mathcal Q}}
\nc{\calR}{{\mathcal R}}
\nc{\calS}{{\mathcal S}}
\nc{\calT}{{\mathcal T}}
\nc{\calU}{{\mathcal U}}
\nc{\calV}{{\mathcal V}}
\nc{\calW}{{\mathcal W}}
\nc{\calX}{{\mathcal X}}
\nc{\calY}{{\mathcal Y}}
\nc{\calZ}{{\mathcal Z}}

\nc{\scrA}{{\mathscr A}}
\nc{\scrB}{{\mathscr B}}
\nc{\scrC}{{\mathscr C}}
\nc{\scrD}{{\mathscr D}}
\nc{\scrE}{{\mathscr E}}
\nc{\scrF}{{\mathscr F}}
\nc{\scrG}{{\mathscr G}}
\nc{\scrH}{{\mathscr H}}
\nc{\scrI}{{\mathscr J}}
\nc{\scrJ}{{\mathscr I}}
\nc{\scrK}{{\mathscr K}}
\nc{\scrL}{{\mathscr L}}
\nc{\scrM}{{\mathscr M}}
\nc{\scrN}{{\mathscr N}}
\nc{\scrO}{{\mathscr O}}
\nc{\scrP}{{\mathscr P}}
\nc{\scrQ}{{\mathscr Q}}
\nc{\scrR}{{\mathscr R}}
\nc{\scrS}{{\mathscr S}}
\nc{\scrT}{{\mathscr T}}
\nc{\scrU}{{\mathscr U}}
\nc{\scrV}{{\mathscr V}}
\nc{\scrW}{{\mathscr W}}
\nc{\scrX}{{\mathscr X}}
\nc{\scrY}{{\mathscr Y}}
\nc{\scrZ}{{\mathscr Z}}

\nc{\bnu}{{\bar{ \nu}}}
\nc{\olO}{\bar{\calO}}
\nc{\pf}{{\on{pf}}}

\nc{\al}{{\alpha}} 
\nc{\be}{{\beta}}
\nc{\ga}{{\gamma}} \nc{\Ga}{{\Gamma}}
\nc{\hGa}{\hat{\Gamma}}
\nc{\ve}{{\varepsilon}} 
\nc{\la}{{\lambda}} \nc{\La}{{\Lambda}}
\nc{\om}{\omega} \nc{\Om}{\Omega} 
\nc{\sig}{{\sigma}} \nc{\Sig}{{\Sigma}}
\nc{\dR}{{\mathrm{dR}}}
\nc{\Perf}{{\mathrm{Perf}}}
\nc{\Gm}{{\mathbb{G}_m}}

\nc{\Spa}{\on{{Spa}}}
\nc{\Spd}{\on{{Spd}}}
\nc{\tnb}{\psi_{\rm tame}}
\nc{\oM}{\overline{{M}}}
\nc{\op}{{\on{op}}}
\nc{\ad}{{\on{ad}}}
\nc{\alg}{{\on{alg}}}
\nc{\Ad}{{\on{Ad}}}
\nc{\Adm}{{\on{Adm}}} \nc{\aff}{{\on{af}}}
\nc{\Aut}{{\on{Aut}}}
\nc{\Bun}{{\on{Bun}}}
\nc{\cha}{{\on{char}}}
\nc{\der}{{\on{der}}}
\nc{\Der}{{\on{Der}}}
\nc{\diag}{{\on{diag}}}
\nc{\End}{{\on{End}}}
\nc{\Fl}{{\calF\!\ell}}
\nc{\Tr}{{\on{Transp}}}
\nc{\TR}{{\calT\!\calR}}
\nc{\Gal}{{\on{Gal}}}
\nc{\Gr}{{\on{Gr}}}
\nc{\Hk}{{\on{Hk}}}
\nc{\rH}{{\on{H}}}
\nc{\Hom}{{\on{Hom}}}
\nc{\IC}{{\on{IC}}}
\nc{\id}{{\on{id}}}
\nc{\Id}{{\on{Id}}}
\nc{\ind}{{\on{ind}}}
\nc{\Ind}{{\on{Ind}}}
\nc{\Lie}{{\on{Lie}}}
\nc{\Pic}{{\on{Pic}}}
\nc{\pr}{{\on{pr}}}
\nc{\Res}{{\on{Res}}}
\nc{\res}{{\on{res}}} \nc{\Sat}{{\on{Sat}}}
\nc{\spc}{{\on{sc}}}
\nc{\drv}{{\on{der}}}
\nc{\sgn}{{\on{sgn}}}
\nc{\Spec}{{\on{Spec}}}\nc{\Spf}{\on{Spf}} 
\nc{\Sph}{\on{Sph}}
\nc{\St}{{\on{St}}}
\nc{\tr}{{\on{tr}}}
\nc{\Mod}{{\mathrm{-Mod}}}
\nc{\Hilb}{{\on{Hilb}}} 
\nc{\Ext}{{\on{Ext}}} 
\nc{\vs}{{\on{Vec}}}
\nc{\ev}{{\on{ev}}}
\nc{\nO}{{\breve{\calO}}}
\nc{\tS}{{\tilde{S}}}
\nc{\spe}{{\on{sp}}}
\nc{\loc}{{\on{loc}}}

\nc{\nscrR}{{\mathscr{R}^{\on{nr}}}}

\nc{\GL}{{\on{GL}}}

\nc{\Gl}{\on{Gl}} 
\nc{\GSp}{{\on{GSp}}}
\nc{\gl}{{\frakg\frakl}}
\nc{\SL}{{\on{SL}}} 
\nc{\SU}{{\on{SU}}} 
\nc{\SO}{{\on{SO}}}
\nc{\PGL}{{\on{PGL}}}

\nc{\Conv}{{\on{Conv}}}
\nc{\Rep}{{\on{Rep}}}
\nc{\Dom}{{\on{Dom}}}
\nc{\red}{{\on{red}}}
\nc{\act}{{\on{act}}}
\nc{\nr}{{\on{nr}}}
\nc{\ctf}{{\on{ctf}}}

\nc{\str}{{\on{-}}} 
\nc{\os}{{\bar{s}}}
\nc{\oeta}{{\bar{\eta}}}

\nc{\hookto}{\hookrightarrow}
\nc{\longto}{\longrightarrow}
\nc{\leftto}{\leftarrow}
\nc{\onto}{\twoheadrightarrow}
\nc{\lonto}{\twoheadleftarrow}

\nc{\pot}[1]{ [\hspace{-0,5mm}[ {#1} ]\hspace{-0,5mm}] }
\nc{\rpot}[1]{ (\hspace{-0,7mm}( {#1} )\hspace{-0,7mm}) }

\numberwithin{equation}{section}

\begin{document}
	
	\title{Singularities of local models}
	
	\author[N. Fakhruddin, T. Haines, J. Louren\c{c}o, T. Richarz]{Najmuddin Fakhruddin, Thomas Haines, Jo\~ao Louren\c{c}o, Timo Richarz}
	
	\address{School of Mathematics, Tata Institute of Fundamental	Research, Homi Bhabha Road, Mumbai 400005, India}
	\email{naf@math.tifr.res.in}
	
	\address{Department of Mathematics, University of Maryland, College Park, MD 20742-4015, USA}
	\email{tjh@umd.edu}

	\address{Mathematisches Institut, Universität Münster, Einsteinstrasse 62, Münster, Germany}
	\email{j.lourenco@uni-muenster.de}
	
	\address{Technische Universit\"at Darmstadt, Department of Mathematics, 64289 Darmstadt, Germany}
	\email{richarz@mathematik.tu-darmstadt.de}
	
	\begin{abstract}
          We construct local models of Shimura varieties and
          investigate their singularities, with special emphasis on
          wildly ramified cases. More precisely, with the exception of
          odd unitary groups in residue characteristic $2$ we
          construct local models, show reducedness of their special
          fiber, Cohen--Macaulayness and in equicharacteristic also
          (pseudo-)rationality. In mixed characteristic we conjecture
          their pseudo-rationality.
                
          This is based on the construction of parahoric group schemes
          over two dimensional bases for wildly ramified groups and an
          analysis of singularities of the attached Schubert varieties
          in positive characteristic using perfect geometry.
	\end{abstract}
	\date{\today}

	\maketitle
	\tableofcontents
	
	\section{Introduction}
	
	\subsection{Background}
	Let $O$ be a complete discretely valued ring with fraction field $K$ and with residue field $k$ of characteristic $p>0$, which for simplicity we assume is algebraically closed.
	Let $G$ be a (connected) reductive group over $K$. 
	
	The local models we consider in this paper are certain flat projective $O$-schemes which model the singularities of integral $O$-models of Shimura varieties (in the case of mixed
	characteristic) and of $G$-shtukas (in the equicharacteristic case) with parahoric level structure. 
	
	Local models attached to PEL type Shimura varieties with parahoric level structure at a given prime number were developed in the book of Rapoport and Zink \cite{RZ96}, and were defined  there in a linear algebra style using moduli spaces of self-dual lattice chains in certain skew-Hermitian vector spaces.  The local models were proved to be \'{e}tale locally isomorphic to the corresponding integral models of the Shimura varieties defined using analogous chains of polarized abelian schemes with additional structure. This has two important consequences:
\begin{enumerate}
\item [(1)] The singularities in the special fiber of the Shimura variety coincide with those of its local model, which can be studied more directly;
\item [(2)] The sheaf of nearby cycles on its special fiber can be determined from the corresponding object on the local model.
\end{enumerate}

The approach in (1) goes back to de Jong \cite{dJ93}, who used it to determine the singularities appearing in Siegel modular 3-folds with Iwahori level structure at $p$ (Shimura varieties attached to ${\rm GSp}(4)$); it was also exploited by many other authors, see \cite{CN90, DP94, Pappas:HilbertBlumenthal, Fal01, Gor01, PR05, PRS13, Pappas_Zachos:OrthogonalShimuraVarieties} for example.  
The method in (2) is a key ingredient in the study of local Hasse--Weil zeta functions of Shimura varieties with parahoric level structure; see the survey articles of Rapoport \cite{Rap90} and the second named author \cite{Hai05, Hai14}.

For more general Shimura varieties (as well as for moduli spaces of
shtukas) a more purely group-theoretic construction of local models
-- also satisfying (1) and (2) above --  is desirable, in parallel
to Deligne's group-theoretic axiomatic construction of Shimura
varieties \cite{Del71b}.
Such constructions also have the benefit of tying the theory of Shimura varieties more closely to Schubert varieties, loop groups, and other objects appearing in the geometric Langlands program. 
This also gives hints about how to make the construction itself, with the help of Beilinson--Drinfeld affine Grassmannians. 

	The sought-after local models, which we denote by $\widetilde{M}_{\calG, \mu}$, arise as the seminormalizations of certain orbit closures $M_{\calG,\mu}$ inside a Beilinson--Drinfeld Grassmannian, and are associated to a parahoric group scheme $\calG$ over $O$ extending $G$, a geometric conjugacy class $\mu$ of cocharacters of $G$ and certain auxiliary additional data in the mixed characteristic case, see \Cref{sec:mixed characteristic local models}. 
	The schemes are constructed by Zhu in \cite{Zhu14} and by
        Pappas--Zhu in \cite{PZ13} for all $G$ splitting over a tamely ramified extension of $K$. 
	Their construction in the mixed characteristic setting is extended by Levin in \cite{Lev16} to all groups $G$ which are restrictions of scalars of tamely ramified groups, so covering all $G$ (up to central isogeny) in the cases where $p \geq 5$. 
	In the equicharacteristic setting, the construction for arbitrary groups is given in \cite{Ric16}. (In all these cases, flatness of the local models so defined is automatic, in contrast with the lattice-theoretic proposals in \cite{RZ96}, which in certain cases failed to be flat, as first pointed out by Pappas.  On the other hand, unlike \cite{RZ96}, these group-theoretic local models are not given by explicit moduli problems.)  
	
	One of the main results of \cite{Zhu14} and \cite{PZ13} is that when $p \nmid | \pi_1(G_{\der})|$ the orbit closures $M_{\calG, \mu}$ are normal (hence coincide with $\widetilde{M}_{\calG, \mu}$) with reduced special fiber, all of whose components are normal, Cohen--Macaulay and compatibly Frobenius split. 
	They also conjecture that under the same conditions the local models are always Cohen--Macaulay \cite[Remark 9.5 (b)]{PZ13}. 
	This is proved by He in \cite{He13} in the case that $G$ is unramified and $\mu$ is minuscule and by the second and fourth named author \cite[Theorem 2.3]{HR22} for $p > 2$ in all cases where local models had been constructed.
	In the case when $p \mid | \pi_1(G_{\der})|$, it is known by
        \cite{HLR24}, that the orbit closures $M_{\calG, \mu}$ are not
        normal in general, so instead one passes to their
        seminormalizations $\widetilde{M}_{\calG, \mu}$ which
        then have the aforementioned properties.  
	
	The paper at hand extends the above results to all $G$ and all
        $p$ with the exception of one family of examples: ramified odd
        unitary groups $G$ in the case $p=2$, see also
        \Cref{wild.ramified.unitary}.  More precisely, excluding this
        family we construct local models $\widetilde{M}_{\calG, \mu}$
        also for wildly ramified groups $G$ which are not necessarily
        restrictions of scalars of tamely ramified groups, and we
        prove that these models are normal, Cohen--Macaulay and have
        reduced special fibers all of whose components are also
        normal, Cohen--Macaulay and compatibly Frobenius split.  The
        reader is referred to \Cref{z.extensions.LM} for the relation
        with the construction of local models via $z$-extensions from
        \cite[Section 2.6]{HPR20}.  Let us now explain our main
        results in more detail.
	
	\subsection{Main results}
	Fix $O\subset K$ with residue field $k$ and $G$ as above.
	Denote by $\Phi_G$ the relative root system of $G$.
	If $G$ ranges through all absolutely simple groups, then $\Phi_G$ is non-reduced if and only if $G$ is an odd unitary group, see \Cref{section_lifts}.
	Our first main result is \Cref{theorem_coherence_allp} in the main text and concerns local models in equicharacteristic:
	
	\begin{theorem}\label{theorem_A} %
		Assume that $K \simeq k\rpot{t}$ has characteristic $p>0$.
		Also assume that $p> 2$ or $\Phi_G$ is
		reduced. Then the local model
		$\widetilde{M}_{\calG, \mu}$ is Cohen--Macaulay, has
		rational singularities, and reduced special fiber
		equal to the admissible locus
		$\widetilde{A}_{\calG,\mu}$.
	\end{theorem}
	
	For the definition of the admissible locus $\widetilde{A}_{\calG,\mu}$, the reader is referred to \Cref{def_admissible}.
	We also note that
	$\widetilde{M}_{\calG, \mu} = {M}_{\calG, \mu}$ when
	$p \nmid | \pi_1(G_{\der})|$, see \Cref{rem:good primes equality} and \Cref{rem: good primes equality mixed}.
	In \Cref{cor_pic_local_model_equal}, we also calculate the Picard group of $\widetilde{M}_{\calG, \mu}$.
	
	Our second main result is \Cref{theorem_coherence_mixed_char} in the main text and concerns mixed characteristic local models:
	
	\begin{theorem}[] 
		\label{theorem_B} %
		Assume that $K$ has characteristic $0$. Also assume
		that $p> 2$ or $\Phi_G$ is reduced. Then the local
		model $\widetilde{M}_{\underline{\calG}, \mu}$ is Cohen--Macaulay and has a
		reduced special fiber equal to the
		$\mu'$-admissible locus $\widetilde{A}_{\calG',\mu'}$. If
		$\widetilde{A}_{\calG',\mu'}$ is irreducible \textup{(}for example, $\calG$ special parahoric\textup{)}, then $\widetilde{M}_{\underline{\calG}, \mu}$ has
		pseudo-rational singularities.
	\end{theorem}
	Here $\calG'$ and $\mu'$ are equicharacteristic analogues of
	$\calG$ and $\mu$ associated to them via a choice of $O\pot{t}$-group lift $\underline{\calG}$, see \Cref{section_lifts}. As
	above, $\widetilde{M}_{\underline{\calG}, \mu} = {M}_{\underline{\calG}, \mu}$ when
	$p \nmid | \pi_1(G_{\der})|$, and see \Cref{cor_pic_local_model_mixed} for its Picard group. \Cref{theorem_B} is
	slightly weaker than \Cref{theorem_A} in that we do not
	prove that the singularities of $\widetilde{M}_{\underline{\calG}, \mu}$ are
	always pseudo-rational. However, we conjecture that this is
	always the case, see \Cref{conjecture_mixed}.

	\subsection{Methods}
	We now explain our methods and the structure of this
	paper. The main input needed to construct the local model (in
	mixed characteristic) is, as in \cite{PZ13}, the construction
	of a lifting of the parahoric group scheme $\calG$ over $O$ to a group scheme $\underline \calG$ over $O\pot{t}$. The
	special fiber of the local model is then a closed subscheme of
	a partial affine flag variety over $k$ and to analyze this we also
	need to construct lifts of parahoric group schemes over
	$k\pot{t}$ to $W(k)\pot{t}$. These steps were carried out for
	tame groups in \cite{PR08, PZ13}.

          So we need to extend these constructions to
        wild groups.  The group lifts are constructed in
        \Cref{section_lifts} using ideas from \cite{Lou23}: we define
        suitable integral models of maximal tori and root groups
        separately which induce birational models and then apply the
        result that such a model extends to a unique group scheme.
        The reason that we have to exclude the case of odd unitary
        groups stems from this very first step since we are unable to
        construct the lifts of root groups in the case of multipliable
        roots when $G$ is ramified and $p=2$, see
        \Cref{wild.ramified.unitary}.

	In \Cref{Fsing_review} we start with a review of $F$-singularities 
	and (pseudo-)rational singularities. These techniques are central to the study  
	of singularities of local models in later sections. 
	\Cref{conjecture_pseudo_rational} states a conjectural mixed characteristic 
	analogue of a result of Schwede and Singh \cite[Appendix A]{HMS14}, which
	would imply pseudo-rationality of mixed characteristic local models, see also the discussion below.

	The first step in analyzing the singularities of local models
        is the study of the singularities of Schubert varieties $S_w$
        in affine flag varieties.  We carry this out in
        \Cref{section_schubert}, first proving in
        \Cref{theorem_frobenius_split_allp} that the
        seminormalizations $\widetilde{S}_w$ are always normal,
        Cohen--Macaulay, compatibly Frobenius split and have rational
        (in fact, even $F$-rational) singularities. We use the by now
        standard method of applying the Mehta--Ramanathan criterion
        for Frobenius splitting, but we need some extra arguments for $p = 2,3$. In
        \Cref{thm_normality_classical_sch_vars}, we then show that if
        $p>2$ or $\Phi_G$ is reduced, then all Schubert varieties
        $S_w$ are normal if and only if $p$ does not divide the order
        of $\pi_1(G_\der)$.
	
	In \Cref{section_lm}, we construct our local models and prove
        our main results. In the equicharacteristic case, the local
        model is canonical. In mixed characteristic, it depends on the
        choice of the group lift $\underline \calG$ constructed in
        \Cref{prop_lifts_witt}. For minuscule $\mu$, which is the case
        relevant to Shimura varieties, it is expected that these are
        independent of all choices, see \cite[Conjecture
        21.4.1]{SW20}, \cite[Conjectures 2.12, 2.15]{HPR20} and also
        \cite{AGLR22}. To identify the special fiber and prove that it
        is reduced, we follow the method of \cite{Zhu14} and
        \cite{PZ13} based on the coherence conjecture. This is fairly
        straightforward, given the results of
        \Cref{section_schubert}. It then remains to prove that the
        special fiber is Cohen--Macaulay. 
        To do
        this, we use a variant of the argument used in \cite[Section
        6]{HR22}, which has the advantage of also being applicable in
        characteristic $2$ since it does not depend on Zhu's global
        Frobenius splitting \cite[Theorem 6.5]{Zhu14}. The proof uses
        some results in commutative algebra by Schwede and Singh
        \cite[Appendix A]{HMS14} to deduce that in equicharacteristic
        $p$ the local models are Cohen--Macaulay and have $F$-rational
        (hence pseudo-rational) singularities.  In the case of mixed
        characteristic, we get the reducedness and Cohen--Macaulayness
        of the special fiber of the local model by comparing with the
        equicharacteristic case. However, it does not seem possible to
        immediately transfer pseudo-rationality from equal
        characteristic to mixed characteristic.  Motivated by this we
        discuss the above mentioned conjectural mixed characteristic
        analogue (\Cref{conjecture_pseudo_rational}) of one of the
        results of Schwede and Singh which, given our other results,
        would suffice to deduce the pseudo-rationality of local models
        in mixed characteristic.

	\subsection{Relationship with the perfectoid theory}
	
	Let us comment on the relationship between this work and the
        other recent works \cite{AGLR22, GL24} by some of the authors.
        The first paper \cite{AGLR22} studied at length a perfectoid
        analogue of the local model constructed in Scholze--Weinstein's book \cite{SW20}.  An important
        conjecture in \cite{SW20} postulated that these perfectoid
        local models, despite only being v-sheaves, should be
        representable by a flat, normal, and projective scheme over
        $O_E$ with reduced special fiber.  This was proved in
        \cite[Section 7]{AGLR22} under \Cref{hyp_odd_unitary} and
        \Cref{not.the.triality}, using the constructions of this paper
        as an input and comparing them to the v-sheaves of perfectoids
        via a specialization principle.  However, we stress that the
        results in \cite{AGLR22} concerning the singularities of local
        models like reducedness of their special fiber and
        Cohen--Macaulayness rely on the present paper.  As for
        \cite{GL24}, it gives a new proof that local models are normal
        with reduced special fiber, including the missing cases of \Cref{hyp_odd_unitary} and
        \Cref{not.the.triality}.  The statements
        in \cite{GL24} related to Frobenius splittings of the special
        fiber or Cohen--Macaulayness rely again on the present paper.

		\subsection{Acknowledgements} 
		
		We thank Johannes Ansch\"utz, Ian Gleason, Stefano Morra, Michael Rapoport, Peter Scholze, and Karl Schwede for helpful conversations and email exchanges.
		Also, we thank the anonymous referee for suggesting several improvements to the exposition of the manuscript. 
		
	N.F.~acknowledges support from the DAE, Government
of India, under Project Identification No.~RTI4001. The research of T.H.~was partially funded by NSF grant DMS 2200873.
This project has received funding (J.L. via Ana Caraiani) from the European Research Council under the European Union’s Horizon 2020 research and innovation program (grant agreement nº 804176), and (J.L.) from the Max-Planck-Institut f\"ur Mathematik. 
This project has received funding (T.R.) from the European Research Council (ERC) under Horizon Europe (grant agreement nº 101040935), funding (T.R.) from the Deutsche Forschungsgemeinschaft (DFG, German Research Foundation) TRR 326 \textit{Geometry and Arithmetic of Uniformized Structures}, project number 444845124 and funding from the LOEWE starting professorship in algebra, project number LOEWE/4b//519/05/01.002(0004)/87.

	\section{Group lifts to two-dimensional bases} \label{section_lifts}
	In our presentation we follow \cite[Sections 2--3]{Lou23} to construct group lifts via gluing from birational group laws.
	The method works for Witt lifts (equicharacteristic) and Breuil--Kisin lifts (mixed characteristic) in the same way which we, however, treat in the separate Sections \ref{section.Witt.lifts} and \ref{subsection_appendix_bk_lift} for readability.
	We start by fixing some notation.
	
	Let $O$ denote a complete discretely valued ring with fraction field $K$ and perfect residue field $k$ of characteristic $p>0$.
	Let $\breve O/O$ be the completion of the maximal unramified extension with fraction field $\breve K/K$.
	Let $G$ be a reductive $K$-group that is quasi-split (automatic if $K=\breve K$ by Steinberg's theorem) and either simply connected or adjoint. Denote by $\breve G:=G\otimes_K \breve K$ the base change.

	Assume $G$ is also almost $K$-simple.  Then
        $G=\Res_{L/K}(G_0)$, for some finite separable field extension
        $L/K$, of an absolutely almost simple $L$-group $G_0$
        \cite[Section 6.21 (ii)]{borel_tits_groupes_reductifs},
        which is necessarily quasi-split
        and simply connected or adjoint, respectively.  Choose a
        separable field extension $M/L$ of minimal degree such that
        $G_0$ splits over its Galois hull.  As the only non-trivial
        automorphism groups of connected
        Dynkin diagrams are $\bbZ/2$ and $S_3$, the extension $M/L$ is
        of degree $\leq 3$.
	
	In this section, we also work under the following:
	
	\begin{hypothesis}\label{hyp_odd_unitary}
		If $p=2$, then the relative root system $\Phi_{\breve G}$ is reduced.
	\end{hypothesis}
	
	An examination of the tables in
	\cite{Tit79} shows that $\Phi_{\breve G}$ is non-reduced if and only if the associated absolutely almost simple group $\breve G_0= G_0\otimes_K \breve K$ is isomorphic to an odd unitary group.
	So Hypothesis \ref{hyp_odd_unitary} excludes this case if $p=2$.

	Fix a maximal $K$-split torus $S \subset G$ with centralizer equal to a maximal torus $T$ and a Borel subgroup $B$ containing it.
	Let $H/\bbZ$ be the split form of $G$ equipped with a pinning. 
	Choose a Chevalley--Steinberg system for $H$, see \cite[Section 2.1]{Lou23}. 
	Let $K^s/K$ be a Galois extension splitting $G$, and fix an isomorphism 
	\begin{equation}\label{equation.splitting.map}
	G\otimes_K K^s\overset{\simeq}{\longto} H\otimes_\bbZ K^s 
	\end{equation}
	preserving the chosen maximal tori and Borel subgroups such that the $\Gal(K^s/K)$-action transported to the target acts by pinned automorphisms, so $G=\Res_{K^s/K}(H\otimes_\bbZ K^s)^{\Gal(K^s/K)}$ by Galois descent.

	The Chevalley--Steinberg system for $H$ induces a Chevalley quasi-system for the quasi-split group $G$ in the sense of \cite[D\'efinition 2.2.6, Proposition 2.2.7]{Lou23}. 
	Essentially, this is the choice of the pair $S\subset B$ in $G$ along with
	a family of isomorphisms
	\begin{equation}\label{equation.quasi.pinning.G}
	x_a\colon U_a \overset{\simeq}{\longto} \begin{cases} \Res_{L_a/K} \bbG_a \\
	\Res_{L_{2a}/K}\bbH_{L_a/L_{2a}} 
	\end{cases}
	\end{equation}
	for all $a \in \Phi_G^{\on{nd}}$ with $\Phi_G^{\on{nd}}\subset \Phi_G$ the subset of non-divisible roots and $U_a$ the corresponding root subgroup.
	Here, if $\Phi_G$ is reduced, then $L_a=M$ if $a \in \Phi_G^<$ is short and $L_a=L$ if $a \in \Phi_G^>$ is long. 
	If $\Phi_G$ is non-reduced, then $L_a=M \supset L=L_{2a}$ if $2a \in \Phi_G$ and $\bbH_{L_a/L_{2a}}$ is the $L_{2a}$-group described in \cite[4.1.9]{BT84}. Here, the quadratic extension $L_a/L_{2a}$ is allowed to be ramified if $p>2$ but must be unramified if $p=2$ by Hypothesis \ref{hyp_odd_unitary}. 
	This induces a Chevalley valuation of $\scrA(G,S,K)$, see
        \cite[4.2.2]{BT84}, which we then regard as the origin of that
        affine space, which then becomes identified with $\scrV(S)=X_*(S)\otimes \bbR$.
	
\begin{remark}\label{wild.ramified.unitary}
  Let us comment on the various hypotheses on
  $G$.
  \begin{enumerate}
  \item If we wished to include the case where $p=2$ and $\Phi_G$ is
    non-reduced, the structure of $U_a$ would be arithmetically more
    involved, particularly as the subset $M^0 \subset M$ of trace zero
    elements does not behave so well, see \cite[Sections 4.1.10, 4.2.20]{BT84}.
 For
    instance, the valuation of $M^0$ divides the quadratic separable
    extensions into those given by roots of primes and the rest of
    them, see \cite[Lemmes 4.3.3, 4.3.4]{BT84}.  Root-of-prime
    extensions are treated in \cite{Lou23} relying on the theory of
    pseudo-reductive groups.  For the other quadratic extensions, we
    do not know, for example, how to construct the groups
    $\underline{U_a}$ that appear below.
  \item The case of quasi-split and simply connected (respectively,
    adjoint) groups $G$ appears to be most important when studying the
    geometry of Schubert varieties and local models.  Note that for
    such $G$ the maximal torus $T$ is induced \cite[Proposition
    4.4.16]{BT84}, which is a technical convenience, see the proof of
    \Cref{lem_equiv_id_apart_witt}.  If we wished to include more
    general central extensions of $G$ with induced maximal torus, we
    could follow the construction in \cite[Section 2.4]{Lou23}, see
    also \Cref{z.extension.sec} for a particular interesting case.
    Further, it should be possible, though difficult, to extend the
    construction of group lifts below to not necessarily quasi-split
    groups using \'etale descent \cite[Section 5]{BT84}.
			
  \end{enumerate}
\end{remark}
	
	\subsection{Witt lifts}\label{section.Witt.lifts}
	In this subsection, we assume that $K$ is a Laurent series field of characteristic $p>0$. 
	Choosing uniformizers $u$ of $L$ and $t$ of $K$, we identify their rings of integers $O_L=k_L\pot{u}$ and $O=k\pot{t}$ as $k$-algebras. 
	The uniformizers satisfy an Eisenstein equation:
	\begin{equation}\label{equation.Eisenstein.Witt}
	u^e+a_{e-1}(t)u^{e-1}+\dots +a_1(t)u+a_0(t)	=0
	\end{equation} 
	where each of the 
	\begin{equation}
	a_i(t)=\sum b_{ij}t^j
	\end{equation} is a power series with $b_{ij}\in k_L$, $b_{i0}=0$ and $b_{01} \neq 0$. 
	Consider now the defining equation
	\begin{equation}\label{equation.Witt.uniformizer}
	u^e+[a_{e-1}(t)]u^{e-1}+\dots +[a_1(t)]u+[a_0(t)] =0
	\end{equation}
	where each of the \begin{equation}[a_i(t)]=\sum [b_{ij}]t^j \end{equation} is a power series in  $W(k_L)\pot{t}$ obtained by taking Teichm\"uller representatives of the coefficients. 
	Then \eqref{equation.Witt.uniformizer} defines the finite free $W(k)\pot{t}$-algebra $W(k_L)\pot{u}$, 
	which reduces modulo $p$ to the $k\pot{t}$-algebra $k_L\pot{u}$. 
	Similarly, we lift $O_M/O_L$ to $W(k_M)\pot{v}/W(k_L)\pot{u}$ via a choice of uniformizers.

	\begin{lemma}\label{lemma.quasi.pinning.generic.extension}
          The quasi-pinned $K$-group
          $(G,B, S,(x_a)_{a\in \Phi_G^{\on{nd}}})$ lifts to
          $(\underline G,\underline B, \underline
          S,(\underline{x_a})_{a\in \Phi_G^{\on{nd}}})$ defined over
          the maximal open subset
          $U\subset \Spec\,W(k)\pot{t}$ over which
          the extension $W(k_M)\pot{v}/W(k)\pot{t}$ is \'etale.
	\end{lemma}
	\begin{proof}
		Firstly, the split form $H/\bbZ$ of $G$ with its Chevalley--Steinberg system induces a split form $H_0/\bbZ$ of $G_0$ with such a system.
		As quasi-pinnings are compatible with restriction of scalars along finite \'etale maps, we reduce to the case $G=G_0$ is absolutely almost simple and without loss of generality also non-split.
		Let $\widetilde V$ be the Galois hull of the finite \'etale map $V:=f^{-1}(U)\to U$ where $f\colon \Spec\, W(k_M)\pot{v}\to \Spec\, W(k)\pot{t}$.
		As $f$ is ramified at $\{t=0\}$, we have $U\subset \Spec\, W(k)\rpot{t}$ and the reduction of $\widetilde V \to U$ modulo $p$ defines a Galois ring extension $\widetilde K/K$ splitting $G$.
		Hence, $\Gal(\widetilde V/U)\to \Gal(\widetilde K/K)$ acts through \eqref{equation.splitting.map} by pinning preserving automorphisms on $H$, replacing $K^s$ by $\widetilde K$ if necessary. 
		We define 
		\begin{equation}\label{equation.generic.extension}
		\underline G = \Res_{\widetilde V/U}(H\otimes_\bbZ \widetilde V )^{\Gal(\widetilde V/U)},
		\end{equation}
		equipped with the quasi-pinning induced from the chosen Chevalley--Steinberg system for $H$, which satisfies the requirements of the lemma.
	\end{proof}

	Note that $W(k)\rpot{t}$ is a Euclidean domain which is not
        local.  Even though the extension
        $W(k_M)\rpot{v}/W(k)\rpot{t}$ is ramified in general, we can
        extend $\underline G$ from $U$ over $\Spec\, W(k)\rpot{t}$ via
        a birational extension process as follows.  Note that we have
        the maximal torus $\underline T$ in $\underline G$ defined
        over $U$.  We consider the family of group schemes consisting
        of the connected N\'eron $W(k)\rpot{t}$-model of
        $\underline T$ denoted by the same symbol, and the unipotent
        group schemes
	\begin{equation}\label{models_laurent_series_witt}
	\underline{U_a} = \begin{cases}\Res_{W(k_a)\rpot{t_a}/W(k)\rpot{t}} \bbG_a\\
	\Res_{W(k_{2a})\rpot{t_{2a}}/W(k)\rpot{t}} \bbH_{W(k_a)\rpot{t_a}/W(k_{2a})\rpot{t_{2a}}}
	\end{cases}
	\end{equation}
	for every non-divisible root $a \in \Phi_G$, extending the quasi-pinning defined in Lemma \ref{lemma.quasi.pinning.generic.extension}. 
	Here, the symbols $k_a$ denote the residue field of the root fields $L_a$, and the variables $t_a$ are either one of the prescribed lifts $u$ or $v$ of the uniformizer of $L_a$, depending on whether it equals $L$ or $M$.
	
	\begin{lemma}\label{lem_laurent_lift_witt}
		The models $(\underline{T}, \underline{U_a})$ glue birationally to a smooth, affine $W(k)\rpot{t}$-group $\underline{G}$ with connected fibers extending \eqref{equation.generic.extension}.
	\end{lemma}
	
	\begin{proof}
		This follows from the method of \cite[Proposition 3.3.4]{Lou23}. 
		Here we give an overview of the argument.
		
		First, we must show that the axioms of \cite[D\'efinition 3.1.1]{BT84} are satisfied:
		These involve showing that the conjugation action of $\underline{T}$ on the $\underline{U_a}$, the commutator morphisms between $\underline{U_a}$ and $\underline{U_b}$ for linearly independent roots, and a rationally defined morphism exchanging the order of $\pm a$ in a rank $1$ big cell extend from $U$ (defined in Lemma \ref{lemma.quasi.pinning.generic.extension}) to all of $\Spec\, W(k)\rpot{t}$. 
		In the rank $1$ case, we can construct $\underline{G}$ explicitly by extending the definition of $\underline G$ over $U$, isogenous to a restriction of scalars of either $\on{SL}_2$ or $\on{SU}_3$, to the more general ring extensions that we consider; this provides us with the first and third morphisms using the N\'eron property of $\underline T$. 
		Hence, the main concern are commutator morphisms. 
		Over the generic fiber, these morphisms are given explicitly in \cite[Section
		A.6]{BT84}, up to sign and conjugation, and only involve natural operations such as sum, multiplication, trace
		and norm, so they are still well-defined over $W(k)\rpot{t}$.
		For example, if $\Phi_G$ is reduced, and $a,b$ are short roots with long sum $c=a+b \in \Phi_G$, then the commutator $\gamma_{a,b}$ is given on points under the fixed pinnings by 
		\begin{equation}
		(x,y) \to \on{tr}_{R[t_a]/R}(xy),
		\end{equation}	
		where $R$ is any $W(k_c)\rpot{t_c}$-algebra, and $x,y \in R[t_a]=R\otimes_{W(k_c)\rpot{t_{c}}} W(k_a)\rpot{t_a}$, up to ignoring sign and conjugation.
		It is now a consequence of \cite[Th\'eor\`eme
                3.2.5]{Lou23} that there is a smooth affine
                $W(k)\rpot{t}$-group $\underline{G}$ with connected
                fibers glued from these closed subgroups. Here, for
                affineness we use the fact that $W(k)\rpot{t}$ is a Dedekind ring.
	\end{proof} 
	
	We already know that $\underline{G}$ is reductive over $k\rpot{t}$ and $K_0\rpot{t}$, where $K_0=W(k)[p^{-1}]$. We can compare a portion of their Bruhat--Tits theory.
	
	\begin{lemma}\label{lem_equiv_id_apart_witt}
		There are identifications
		\begin{equation}\label{equation.apartment.identification}
		\scrA(\underline{G},\underline{S},k\rpot{t}) \simeq \scrA(\underline{G},\underline{S}, K_0\rpot{t}),
		\end{equation}
		of apartments, equivariant along a natural identification of the Iwahori--Weyl groups.
	\end{lemma}

	\begin{proof}
		Our method of proof is similar to \cite[Proposition 3.4.1]{Lou23}. 
		We fix as origin of the apartments the Chevalley--Steinberg valuations determined by the quasi-pinning inherited from \eqref{models_laurent_series_witt}. 
		Then, both identify with the real vector space $\scrV(\underline{S})$ generated by the coweights of the split torus $\underline{S}$ compatibly with the hyperplanes. 
		
		As $G$ is assumed to be either simply connected or adjoint, the maximal torus $T$ is induced and so is $\underline T$ over $U$.
		We denote by $\underline{\calT}$ its connected N\'eron $W(k)\pot{t}$-model, see \cite[D\'efinition 3.3.3]{Lou23} and \cite[Part IV, Proposition 3.8]{Lou20}.
		Let $\underline{N}$ be the normalizer of $\underline{S}$ in $\underline{G}$.
		In order to identify the Iwahori--Weyl groups, we prove that they are isomorphic to
		\begin{equation}\label{equation.global.IW}
		\underline{N}(W(k)\rpot{t})/\underline{\calT}(W(k)\pot{t})
		\end{equation}	
		via the natural maps as follows.  It suffices to show
                that the natural maps
		\begin{equation}\label{equation2.lem_equiv_id_apart_witt}
		\underline{T}(W(k)\rpot{t})/\underline{\calT}(W(k)\pot{t}) \to \underline{T}(L\rpot{t})/\underline{\calT}(L\pot{t})
		\end{equation}
		and 
		\begin{equation}\label{equation3.lem_equiv_id_apart_witt}
		\underline{N}(W(k)\rpot{t})/\underline{T}(W(k)\rpot{t}) \to \underline{N}(L\rpot{t})/\underline{T}(L\rpot{t}),
		\end{equation}
		are isomorphisms, where $L$ equals either $k$ or $K_0$. 
		The first case \eqref{equation2.lem_equiv_id_apart_witt} is verified by decomposing $\underline{\calT}$ as a product of restriction of scalars of multiplicative group schemes. 
		The second case \eqref{equation3.lem_equiv_id_apart_witt} is a consequence of constancy of the Weyl group of a split torus and vanishing of $H^1$ for $\underline{T}$. 
		One sees readily that these comparison isomorphisms are compatible with those of the apartments and the corresponding group actions.
              \end{proof}

	For any point $x $ in the apartments, we have a certain optimal quasi-concave function $f_x\colon \Phi_G \to \bbR$ in the sense of \cite[Section 4.5]{BT84}, defined with respect to the chosen origin, the Chevalley--Steinberg valuation. We use this to define the $W(k)\pot{t}$-models $\underline{\calU_{a,x}}$ via
	\begin{equation}\label{equation.root.group.extension.Witt}
	\underline{\calU_{a,x}}=\begin{cases} \Res_{W(k_a)\pot{t_a}/W(k)\pot{t}} \big(t_a^{e_af_x(a)}\bbG_{a}\big)\\
	\Res_{W(k_{2a})\pot{t_{2a}}/W(k)\pot{t}} \big( t_a^{(e_af_x(a),e_af_x(2a))}\bbH_{W(k_a)\pot{t_{a}}/W(k_{2a})\pot{t_{2a}}}\big)
	\end{cases},
	\end{equation}
	where the $e_a$ are the ramification degrees of the root field extension $L_a/K$, and by construction the $e_af_x(a)$ are integers.
	
	\begin{proposition}\label{prop_lifts_witt}
		The models $\underline{\calT}$ and $\underline{\calU_{a,x}}$ for all $a \in \Phi_G^{\on{nd}}$ birationally glue to a smooth, affine $W(k)\pot{t}$-group scheme $\underline{\calG_{x}}$ with connected fibers. 
		Its reductions to $k\pot{t}$ and $K_0\pot{t}$ are parahoric group schemes coming from facets which correspond under \eqref{equation.apartment.identification}.
	\end{proposition}
	
	\begin{proof}
          To see that the models $\underline{\calT}$ and $\underline{\calU_{a,x}}$ for all $a \in \Phi_G^{\on{nd}}$
          satisfy the axioms of \cite[D\'efinition 3.1.1]{BT84}, we can proceed as in \cite[Proposition 3.4.5]{Lou23}: due to the equality $W(k)\rpot{t}\cap K_0\pot{t}=W(k)\pot{t}$, it suffices to apply \cite[Th\'eor\`eme 3.8.1]{BT84} to prove the existence of a birational group law. So it glues to a smooth and separated group scheme $\underline{\calG_{x}}$ with connected fibers due to \cite[Th\'eor\`eme 3.2.5]{Lou23}. 
		
		This group scheme is quasi-affine and admits a smooth affine hull, whose geometric fibers are connected outside the unique closed point of $\Spec(W(k)\pot{t})$, by \cite[Proposition 3.2.7]{Lou23}. 
		In order to check affineness, we apply verbatim the proof in \cite[Th\'eor\`eme 3.4.10]{Lou23}: indeed, this relies on the identification of the Iwahori--Weyl groups given in \Cref{lem_equiv_id_apart_witt}.
	\end{proof}	
	
	\subsection{Breuil--Kisin lifts}\label{subsection_appendix_bk_lift}
	In this subsection, we assume that $K$ has characteristic
        zero.  So $L/K$ is a finite extension of complete discretely
        valued fields of characteristic zero with perfect residue
        fields $k_L/k$ of characteristic $p>0$. Define also
        $L^{\mathrm{nr}}/K$ as the maximal unramified subextension of
        $L/K$, so $L/L^{\mathrm{nr}}$ is totally ramified.  Choosing
        uniformizers $\pi_L$ of $L$ and $\pi_K$ of $K$, they satisfy
        an Eisenstein equation:
	\begin{equation}
	\pi_L^e+a_{e-1}(\pi_K)\pi_L^{e-1}+\dots +a_0(\pi_K) =0,
	\end{equation} 
	where each of the $a_i(\pi_K) \in \pi_KO_{L^{\mathrm{nr}}}$ is a power
        series in $\pi_K$ with coefficients being Teichm\"uller
        representatives of elements in $k_L$
        and satisfying the usual constraints, compare with
        \eqref{equation.Eisenstein.Witt}.  Assume without loss of
        generality that there exists $i$ with $(i,p)=1$ and
	\begin{equation}\label{equation.separability.trick}
	a_i(\pi_K) \neq 0. 
	\end{equation}
	This can be achieved by replacing $\pi_L$ by $\pi_L+\pi_K$, if needed.
	Consider, in analogy to \eqref{equation.Witt.uniformizer}, the equation
	\begin{equation}\label{equation.Teichmueller.mixed}
	u^e+[a_{e-1}(t)]u^{e-1}\dots +[a_0(t)]=0,
	\end{equation}
	where $u$ and $t$ are indeterminates, each of the $[a_i(t)] \in W(k_L)\pot{t}$ is obtained from $a_i(\pi_K)$ by taking the coefficients and by replacing $\pi_K$ by $t$. 
	Equation \eqref{equation.Teichmueller.mixed} defines the finite free $W(k)\pot{t}$-algebra $W(k_L)\pot{u}$.	 
	We repeat this procedure for $M/L$: choose a uniformizer $\pi_M$ of $M$ satisfying the analogue of \eqref{equation.separability.trick} with respect to $\pi_L$, an indeterminate $v$ and define the finite free $W(k_L)\pot{u}$-algebra $W(k_M)\pot{v}$.
	Tensoring with $O$ over $W(k)$, we arrive at the finite free ring extensions
	\begin{equation}\label{equation.mixed.lifts}
	O\pot{t} \subset O_{L^{\on{nr}}}\pot{u} \subset O_{M^{\on{nr}}}\pot{v},
	\end{equation}
	where $L^{\on{nr}}$ and $M^{\on{nr}}$ are the maximal unramified subextensions of $L/K$ and $M/K$, respectively. 
	The tower \eqref{equation.mixed.lifts} reduces modulo $t-\pi_K$ to $O\subset O_L\subset O_M$; its reduction modulo $\pi_K$ is $k\pot{t}\subset k_L\pot{u} \subset k_M\pot{v}$ with separable fraction field extensions by \eqref{equation.separability.trick}.

	As for the group $G$ endowed with its quasi-pinning $(B,S,(x_a)_{a\in \Phi_G^{\on{nd}}})$, these data also lift to an open neighborhood $U\subset \Spec\,O\pot{t}$ of the points $(\pi_K)$ and $(t-\pi_K)$ in analogy to \Cref{lemma.quasi.pinning.generic.extension}, and we denote the resulting $U$-groups by $\underline G$, $\underline B$, $\underline T$ and $\underline S$ as before.
	To extend $\underline G$ from $U$ over $\Spec\, O\rpot{t}$, we proceed again via a gluing procedure using extensions of birational group laws. 
	Consider the family of group schemes consisting of the connected N\'eron $O\rpot{t}$-model $\underline{T}$ (note that $O\rpot{t}$ is a Dedekind domain), and the unipotent group schemes
	\begin{equation}\label{models_laurent_series_bk}
	\underline{U_a} = \begin{cases}\Res_{O_{L_a^{\rm{nr}}}\rpot{t_a}/O\rpot{t}} \bbG_a\\
	\Res_{O_{L_{2a}^{\rm{nr}}}\rpot{t_{2a}}/O\rpot{t}} \bbH_{O_{L_a^{\rm{nr}}}\rpot{t_a}/O_{L_{2a}^{\rm{nr}}}\rpot{t_{2a}}}
	\end{cases}
	\end{equation}
	for every root $a \in \Phi_G^{\on{nd}}$, extending the generic quasi-pinning. 
	Here, the variables $t_a$ are either $t$, $u$ or $v$ depending on the cases for the root fields $L_a$ for $a\in \Phi_G$ explicated in \eqref{equation.quasi.pinning.G}.
	We arrive at the following result:

	\begin{lemma}\label{lem_equiv_id_apart_bk}
		The models $(\underline{T}, (\underline{U_a})_{a\in \Phi_G^{\on{nd}}})$ birationally glue to a smooth, affine $O\rpot{t}$-group $\underline{G}$ with connected fibers.
		Furthermore, its fibers over $K$ and $\kappa\rpot{t}$ for $\kappa=k,K$ are reductive, and there are identifications of apartments
		\begin{equation}\label{equation.identification.apartments.bk}
		\scrA(G,S,K) \simeq \scrA(\underline{G},\underline{S}, \kappa\rpot{t}),
		\end{equation}
		equivariantly for the respective Iwahori--Weyl groups. 
	\end{lemma}
	\begin{proof}
		The proofs of \Cref{lem_laurent_lift_witt} and \Cref{lem_equiv_id_apart_witt} translate literally.
	\end{proof}

	For any point $x$ in the apartments \eqref{equation.identification.apartments.bk}, we have the quasi-concave function $f_x\colon \Phi_G \to \bbR$, compare with \eqref{equation.root.group.extension.Witt}. 
	We define the $O\pot{t}$-models $\underline{\calU_{a,x}}$ by 
	\begin{equation}\label{equation.root.models}
	\underline{\calU_{a,x}}=\begin{cases} \Res_{O_{L_a^{\rm{nr}}}\pot{t_a}/O\pot{t}}\big( t_a^{e_af_x(a)}\bbG_{a}\big)\\
	\Res_{O_{L_{2a}^{\rm{nr}}}\pot{t_{2a}}/O\pot{t}} \big(t_a^{(e_af_x(a),e_af_x(2a))}\bbH_{O_{L_a^{\rm{nr}}}\pot{t_{a}}/O_{L_{2a}^{\rm{nr}}}\pot{t_{2a}}}\big)
	\end{cases},
	\end{equation}
	where the $e_a$ are the ramification degrees of the extensions $L_a/K$ and by construction the $e_af_x(a)$ are integers.
	Let $\underline \calT$ be the connected N\'eron $O\pot{t}$-model of the induced torus $\underline T_U$, compare with the proof of \Cref{lem_equiv_id_apart_witt}.
	
	\begin{proposition}\label{prop_lifts_bk}
		The models $\underline{\calT}$ and $\underline{\calU_{a,x}}$ for all $a \in \Phi_G^{\on{nd}}$ birationally glue to a smooth, affine $O\pot{t}$-group scheme $\underline{\calG_{x}}$ with connected fibers. 
		Its reductions to $O$ and $\kappa\pot{t}$, with $\kappa=k,K$ are parahoric group schemes coming from facets which correspond under \eqref{equation.identification.apartments.bk}.
	\end{proposition}
	
	\begin{proof}
		The proof of \Cref{prop_lifts_witt} applies verbatim.
	\end{proof}

	\section{A conjecture on pseudo-rationality} \label{Fsing_review}
	
	We recall some definitions and facts from the theory of singularities, especially in positive characteristic. 
	\Cref{conjecture_pseudo_rational} below is a mixed characteristic analogue of a result of Schwede--Singh recalled in \Cref{lemma.Schwede.Singh}. 
	Its proof would imply that mixed characteristic local
        models also have pseudo-rational singularities, see \Cref{conjecture_mixed}. 
	
	\subsection{Review of $F$-singularities}
	A Noetherian scheme $X$ over $\bbF_p$ is
	said to be $F$-\emph{finite} if the absolute Frobenius
	morphism $F\colon X \to X$ is a finite morphism (for example, finite type
	schemes over $F$-finite fields). It is said to be $F$-\emph{split} if the canonical morphism $\mathcal O_X \rightarrow F_*\mathcal O_X$ has an $\mathcal O_X$-linear splitting. We say $X$ is \emph{stably $F$-split} if for some $e>0$ the map $\mathcal O_X \to F^e_*\mathcal O_X$ splits, and the two notions are equivalent by \cite[Lemma 5.0.3]{BS13}. Moreover, a closed subscheme $Y\subset X$ is \emph{compatibly (stably) $F$-split} if the corresponding splittings respect the closed immersion, and again the stable notion is an equivalent one by \cite[Lemma 6.0.4]{BS13}.	 
	A local $\bbF_p$-algebra
	$(R,\frakm)$ is said to be $F$-\emph{injective} if the map on
	local cohomology
	$F_*\colon H^{\bullet}_{\frakm}(R) \to H^{\bullet}_{\frakm}(R)$ is
	injective (for example, local rings of $F$-split schemes).

	A Noetherian reduced $F$-finite $\bbF_p$-algebra $R$ is said to be $F$-\emph{regular} if every prime ideal localization $R_{\mathfrak p}$ has all its ideals tightly closed, see \cite[Section 1]{HH89}. 
	If every parameter ideal of such an $R_{\mathfrak p}$ is tightly closed, we say $R_{\mathfrak p}$ is $F$-\emph{rational}; see \cite[Definition 4.1]{HH94a}, and also \cite{FW89} or \cite{Smi97}. We say a Noetherian reduced $F$-finite $\bbF_p$-scheme has $F$-\emph{rational singularities} if all of its local rings are $F$-rational. 
	
	A projective scheme $X$ over an $F$-finite field is said to be \emph{globally} $F$-\emph{regular} provided that for every ample invertible sheaf $\mathcal L$, the section ring $\bigoplus_{n \in \mathbb Z_{\geq 0}} H^0(X, \mathcal L^{\otimes n})$ is a \emph{strongly} $F$-\emph{regular} ring, in the sense of \cite[Section 3]{HH89} (see also \cite[Definition 5.2]{Cas22}). 
	By \cite[Theorem 3.1(d)]{HH89}, any strongly $F$-regular ring is $F$-regular (the converse is expected but appears to be an open question in general). 
	A key property of strong $F$-regularity is that it passes to all prime localizations of the ring.

	We shall use the following results, extracted from \cite{HH89, Smi00, HMS14}.

	\begin{lemma} \label{Smith00}
		A globally $F$-regular projective variety ${\rm Proj}(S)$ over a perfect field is $F$-rational.
	\end{lemma}
	
	\begin{proof} 
		By \cite[Theorem 3.10]{Smi00}, $S$ is strongly $F$-regular.  
		By \cite[Theorem 3.1]{HH89}, all localizations of $S$ and all direct summands of such are strongly $F$-regular.  
		This means the local rings of ${\rm Proj}(S)$ are strongly $F$-regular.  
		Now by \cite[Theorem 3.1(d)]{HH89}, they are also $F$-regular, which means that all ideals are tightly closed.  
		In particular these local rings are $F$-rational.
	\end{proof}

	\begin{lemma}\label{lemma.Schwede.Singh}
          Let $R$ be an $F$-finite Noetherian local ring and $t$
          a non-zero divisor.  If $R/(t)$ is $F$-injective and
          $R[t^{-1}]$ is $F$-rational, then $R$ is $F$-rational.
	\end{lemma}
	\begin{proof}
	This is Schwede--Singh \cite[Corollary A.4]{HMS14}.
	\end{proof}

	\subsection{(Pseudo-)rational singularities}
	We follow \cite[Section 2]{LT81} (see also \cite[Definition 1.8]{Smi97}). An excellent (thus Noetherian) local ring $(R,\mathfrak m)$ is said to be \emph{pseudo-rational} if it is normal, Cohen--Macaulay, admits a dualizing complex, and if for each proper birational morphism $\pi\colon Y \rightarrow {\rm Spec}(R)$ with $Y$ normal, the canonical map 
	\begin{equation}\label{equation.definition.pseudorational}
	f_*\omega_Y\to \omega_R
	\end{equation}
	is an isomorphism (or equivalently, is surjective on global sections \cite[Section 4]{LT81}, or equivalently by duality theory $H^{d}_\frakm(R)\to \bbH^d_{\frakm}(Rf_*\calO_Y)=H^{d}_{f^{-1}(\frakm)}(\calO_Y)$ is injective for $d={\rm dim}(R)$).
	An excellent scheme has \emph{pseudo-rational singularities} (or is \emph{pseudo-rational}) if each of its local rings is pseudo-rational.
	
	\begin{remark}
	In order to establish pseudo-rationality, one may restrict to the class of \textit{projective} birational morphisms $\pi\colon Y \rightarrow {\rm Spec}(R)$ with $Y$ normal, by an application of Chow's lemma.
	Further, we note that the definition of pseudo-rationality in \cite[Definition 2.6]{MS20} is weaker in that $R$ is not required to be excellent or normal.
	\end{remark}
	
	\begin{lemma}  \label{Smith97}
		Any excellent local $\bbF_p$-algebra $R$ which is $F$-rational is also pseudo-rational.
	\end{lemma}
	\begin{proof}
		This is \cite[Theorem 3.1]{Smi97}.
	\end{proof}
	
	The next lemma is used in \Cref{theorem_coherence_mixed_char} to establish pseudo-rationality of special local models:

	\begin{lemma}\label{lemma.pseudorational}
          Let $R$ be a local ring of mixed characteristic $(0,p)$
          which is excellent, normal and admits a dualizing complex.
          Let $\pi \in \frakm$ be a non-zero divisor such that
          $R/(\pi)$ is an $\bbF_p$-algebra.  If $R/(\pi)$ is
          $F$-rational, then $R$ is pseudo-rational.
	\end{lemma}
	\begin{proof}
	Since $R$ is assumed to be normal, this is \cite[Theorem 3.8]{MS20}. 
	\end{proof}
  
	The following conjecture is a mixed characteristic analogue of \Cref{lemma.Schwede.Singh}. 
	Since it does not appear in the
	literature (but see the discussion at \url{https://mathoverflow.net/q/396462}), 
	we write it down here:
	
	\begin{conjecture} \label{conjecture_pseudo_rational}  
	In the situation of \Cref{lemma.pseudorational}, if $R/(\pi)$ is $F$-finite and
		$F$-injective, and $R[\pi^{-1}]$ is pseudo-rational, then $R$ is pseudo-rational.
	\end{conjecture}
	
        We conclude this section by recalling a stronger notion of
        rationality over a perfect field $k$.  Let $X$ be a finite
        type $k$-scheme. We say $X$ has {\em rational singularities}
        if it is Cohen-Macaulay and there exists a proper birational
        morphism $f\colon Y \rightarrow X$ of $k$-schemes with $Y$
        smooth over $k$ (in which case we say $X$ has a {\em
          resolution of singularities}) such that the natural map
        $\mathcal{O}_X \to Rf_* \mathcal{O}_Y$ is an isomorphism.  It
        follows using Grothendieck--Serre duality that if $X$ has
        rational singularities then $R^if_*\omega_Y=0$ for all $i>0$. Moreover, it also follows that $X$ has pseudo-rational
        singularities by \cite[Lemma 9.3]{Kov17}, but we do not use this fact anywhere in the present paper.
        We note that this notion of rational singularities is independent of the
        choice of resolution by \cite[Theorem 1]{CR11}.

	\section{Schubert varieties} \label{section_schubert}
	Let $k$ be an algebraically closed field of characteristic $p>0$, $K=k\rpot{t}$ be the corresponding Laurent series field and $O=k\pot{t}$ the power series ring. 
	
	Let $G$ be a reductive $K$-group. 
	For each facet $\bbf\subset \scrB(G,K)$ of the Bruhat--Tits building, we denote by $\calG=\calG_\bbf$ the associated parahoric $O$-group scheme extending $G$, see \cite[Définition 5.2.6 ff.]{BT84}. 
	
	The loop group $LG$, respectively positive loop group $L^+\calG$, is the functor on the category of $k$-algebras $R$ given by $LG(R)=G(R\rpot{t})$, respectively $L^+\calG(R)=\calG(R\pot{t})$.
	Then $L^+\calG\subset LG$ is a subgroup functor, and the {\it \textup{(}twisted partial\textup{)} affine flag variety} is the \'etale quotient
	\begin{equation}
	\Fl_{\calG}=LG/L^+\calG,
	\end{equation}
	which is represented by an ind-projective $k$-ind-scheme by \cite[Theorem 1.4]{PR08}. 
	
	In the following, we fix two facets $\bbf,\bbf'\subset \scrB(G,K)$ and denote by $\calG=\calG_\bbf$, $\calG'=\calG_{\bbf'}$ the associated parahorics.
	Given an element $w\in L^+\calG'(k)\backslash LG(k)/L^+\calG(k)$, 
	the {\it Schubert variety} $S_w$ is the reduced $L^+\calG'$-orbit closure of $\tilde{w}\cdot e$ in $\Fl_{\calG}$,
	where $\tilde w\in LG(k)$ is any representative of $w$
 	and $e$ the base point of $\Fl_{\calG}$, see \cite[Definition 8.3]{PR08} and compare with \cite[Section 3]{HR22}.
	Then $S_w$ is a projective $k$-variety admitting the $L^+\calG'$-orbit $C_w$ of $\dot w\cdot e$ as a dense open subset.
	This induces a presentation on reduced ind-schemes
	\begin{equation}\label{ind.presentation.aff.flag}
	(\Fl_{\calG})_{\red}=\on{colim} S_w,
	\end{equation}
	where $w$ runs through the double cosets as above, and all transition maps $S_v\to S_w$ are closed immersions.

	\subsection{$F$-singularities of seminormalized Schubert varieties}
	Let $\widetilde{S}_w\to S_w$ be the seminormalization \cite[0EUK]{StaProj}, that is, the initial scheme mapping universally homeomorphically to $S_w$ with the same residue fields.
	In this subsection we show the following result for general reductive $K$-groups:
	
	\begin{theorem}\label{theorem_frobenius_split_allp}
		The seminormalized Schubert varieties $\widetilde{S}_w$ are
		normal, Cohen--Macaulay, compatibly $F$-split and have rational
		singularities. Furthermore, the $\widetilde{S}_w$ are globally
		$F$-regular, hence have $F$-rational
		singularities.
	\end{theorem}

	Here {\it compatibly} $F$-split carries the following meaning.
	By functoriality of seminormalizations \cite[0EUS]{StaProj}, there are maps $\widetilde{S}_v \to \widetilde{S}_w$ lifting the closed immersions $S_v \to S_w$ from \eqref{ind.presentation.aff.flag}, yielding the (a priori non-strict) ind-scheme
	\begin{equation}\label{normalized.ind.scheme}
	\widetilde{\Fl}_{\calG}=\on{colim}\widetilde{S}_w.
	\end{equation}
	In the course of the proof of \Cref{theorem_frobenius_split_allp}, we show that $\widetilde{S}_v \to \widetilde{S}_w$ are closed immersions (see \Cref{lemma.reduction.main.theorem}) and that $\widetilde{S}_w$ is $F$-split compatibly with all closed subvarieties $\widetilde S_v$.

	\begin{remark}
		The methods from \cite[Theorem 8]{Fal03}, \cite[Theorem 8.4]{PR08} and \cite[Theorem 1.4]{Cas22} essentially imply \Cref{theorem_frobenius_split_allp} for all groups whose adjoint simple factors are Weil restrictions of scalars of tamely ramified groups. 
		\Cref{theorem_frobenius_split_allp} is new whenever one of the absolutely simple factors is wildly ramified, therefore covering general reductive $K$-groups.
	\end{remark}
	
\begin{remark} 
  There exist surfaces which have rational, but not $F$-rational,
  singularities \cite[Example 2.11]{HW96}.  Further, we note that by
  the proof of Lemma \ref{Smith00}, we know something slightly
  stronger than $F$-rationality, namely, the local rings of
  $\widetilde{S}_w$ are $F$-regular.
\end{remark}

	\subsubsection{Preliminary reductions for the proof of \Cref{theorem_frobenius_split_allp}}
	Recall the notation $\calG=\calG_\bbf$, $\calG'=\calG_{\bbf'}$.
	Let $S$ be a maximal $K$-split torus with $\bbf,\bbf'\subset \scrA(G,S,K)$, see \cite[Theorem 7.4.18 (i)]{BT72}.
	Fix an alcove $\bba$ in the apartment containing $\bbf$ in its closure, and denote by $\calI=\calG_{\bba}$ the associated Iwahori $O$-group scheme. 
	The affine Weyl group $W_{\on{af}}$ (respectively, its subgroup $W_\calG$) is the Coxeter group generated by the simple reflections along the hyperplanes meeting the closure of $\bba$ (respectively, passing through $\bbf$).
	There is a natural bijection $W_{\on{af}}/W_\calG\cong L^+\calI(k)\backslash LG^0(k)/L^+\calG(k)$ where $LG^0$ denotes the neutral component.
	In order to prove \Cref{theorem_frobenius_split_allp}, we may and do assume without loss of generality that $\calG'=\calI$ and $w\in W_{\on{af}}/W_\calG$, as every Schubert variety is isomorphic to one of this particular form by \cite[Section 3.1, Corollary 3.2]{HR22}.

	In the following we identify the Bruhat order on the coset space $W_{\on{af}}/W_\calG$ compatibly with the Bruhat order on the subset of right $W_\calG$-minimal representatives in $W_{\on{af}}$, see \cite[Lemma 1.6]{Ri13}. 		Suppose $w \in W_{\on{af}}$ is right $W_\calG$-minimal.
	Fix a reduced decomposition as a product of simple reflections $\dot w = s_{1}\cdot \ldots \cdot s_{d}$ in $W_{\on{af}}$.
	Denote by $D_{\dot w}$ the Demazure variety for $\dot w$, denoted $D(\tilde w)$ in \cite[Proposition 8.8]{PR08}. 
	By \cite[Section 3.3]{HR22}, there is a projective morphism 
	\begin{equation}\label{equation.reduced.decomposition.Schubert}
	D_{\dot{w}} \to S_w,
	\end{equation}
	which is an isomorphism over the open Schubert cell $C_w$, hence birational and surjective. 
	For any $v\leq w$ in the Bruhat order, the reduced decomposition $\dot w$ induces a (not necessarily unique) reduced decomposition $\dot v$ of $v$, so there exists a closed immersion $D_{\dot v}\to D_{\dot w}$ covering $S_v\to S_w$.
	The following lemma makes the connection to the normalized Schubert varieties appearing in \cite{HLR24, HR22}: 
	
	\begin{lemma}\label{seminormal.Schubert.varieties.lemma}
	The seminormalized Schubert varieties $\widetilde S_w$ are normal.
	\end{lemma}
	\begin{proof}
	The normalization morphism $S_w^{\on{nor}}\to S_w$ is a universal homeomorphism \cite[Proposition 3.1 i)]{HR22}, which induces an isomorphism over $C_w$ (because it is regular).
	By the universal property of seminormalizations, it remains to show that $S_w^{\on{nor}}\to S_w$ induces an isomorphism on {\it all} residue fields.  
	We observe that there are transition maps $S_v^{\on{nor}} \to S_w^{\on{nor}}$ lifting the closed immersions $S_v \to S_w$, see the proof of \cite[Proposition 9.7 (b)]{PR08} using the functoriality of the Demazure resolution \eqref{equation.reduced.decomposition.Schubert}.
	Now, given a point $x\in S_w$ lying in some cell $C_v$, it induces a tower of residue field extensions $\kappa(S_v^{\on{nor}},x) \supset \kappa(S_w^{\on{nor}},x) \supset \kappa(S_w,x)=\kappa(C_v,x)$. 
	As $S_v^{\on{nor}}\to S_v$ is an isomorphism over $C_v$, all inclusions are equalities which implies the lemma.
	\end{proof}
	
	\Cref{seminormal.Schubert.varieties.lemma} implies that \eqref{equation.reduced.decomposition.Schubert} factors through $\widetilde S_w\to S_w$ inducing the birational projective morphism
	\begin{equation}\label{equation.reduced.decomposition.Demazure}
	f\colon D_{\dot w}\to \widetilde S_w,
	\end{equation}
	with the property $f_{*}\calO_{D_{\dot w}}=\calO_{\widetilde S_w}$, compare \cite[Section 3.3]{HR22}. 
	The proof of the next lemma follows the arguments from \cite{LRPT06, Cas22} and reduces \Cref{theorem_frobenius_split_allp} to the corresponding result for Demazure varieties. 
	The latter case is proved in \Cref{subsection_demazure}, see \eqref{eq.finish.proof} for details.

	\begin{lemma}\label{lemma.reduction.main.theorem}
	Assume that the Demazure variety $D_{\dot w}$ is compatibly stably $F$-split 
	with the divisors $D_{\dot v}$ for all subwords $\dot v$ of $\dot w$ of colength $1$.
	Then \Cref{theorem_frobenius_split_allp} holds true.
	\end{lemma}
	\begin{proof}
	Recall that compatibly stably $F$-split varieties are compatibly $F$-split by \cite[Lemmas 5.0.3, 6.0.4]{BS13}.
	Now, if $D_{\dot w}$ is compatibly $F$-split with the divisors $D_{\dot v}$ for $\dot v$ of colength $1$, then $D_{\dot w}$ is compatibly $F$-split both with their union $\partial D_{\dot w}$ and $D_{\dot v}$ for all subwords $\dot v$ of $\dot w$ by \cite[Proposition 1.2.1]{BK07}.
	
	Compatibility with $\partial D_{\dot w}$ implies that $D_{\dot w}$ is globally $F$-regular (following the second part of the argument in \cite[Proposition 5.8]{Cas22} which applies verbatim). Recall that we have an identity $f_*\calO_{D_{\dot w}}=\calO_{\widetilde{S}_w}$ where $f$ denotes the map \eqref{equation.reduced.decomposition.Demazure}, compatibly with the Frobenius, which allows us to descend any $F$-splitting along the proper cover $f$. More generally, we can apply \cite[Lemma 1.2]{LRPT06} to $f$ and deduce $\widetilde S_w$ is globally $F$-regular. \Cref{Smith00} implies that $\widetilde{S}_w$ has $F$-rational singularities.
	 Then by \Cref{Smith97}, $\widetilde{S}_w$ is pseudo-rational, and in particular, is Cohen--Macaulay. 
	 
	 Next, consider the scheme-theoretic image $T_{v,w}$ of the map $\widetilde S_v\to \widetilde S_w$ and follow the argument in \cite[Proposition 9.7 (b)]{PR08}. This is a $k$-variety with seminormalization equal to $\widetilde{S}_v$. Since $D_{\dot v}$ is compatibly $F$-split inside $D_{\dot{w}}$, we deduce that $T_{v,w}$ is also compatibly $F$-split with $\widetilde{S}_w$ by pushforward along the map $f$. But $F$-split schemes are weakly normal by \cite[Proposition 1.2.5]{BK07}, and in particular seminormal, so we have that $T_{v,w}\simeq \widetilde{S}_v$. In other words, the maps $\widetilde{S}_v\to \widetilde{S}_w$ are closed immersions for all $v\leq w$ and compatibly $F$-split. 
	 
	 Finally, we handle rationality of $\widetilde{S}_w$. We factor $f$ as partial Demazure resolutions having fibers of dimension at most $1$
	 \begin{equation}f_i \colon D_{\dot{u}_i}\tilde{\times}\widetilde{S}_{v_i}\to D_{\dot{u}_{i+1}}\tilde{\times}\widetilde{S}_{v_{i+1}},\end{equation} where we write $\dot w=\dot u_i \cdot \dot v_i$ with $\dot u_i=s_1\cdot \dots \cdot s_{d-i}$ and $\dot v_i=s_{d-i+1}\cdot \dots \cdot s_d$.
	 By induction, it suffices to show vanishing of the higher direct images of the structure sheaf along $f_i$. Moreover, we may even ignore the factor $D_{\dot{u}_{i+1}}\subset D_{\dot{u}_i}$ and reduce to the study of $g\colon S_s \tilde{\times}\widetilde{S}_v \to \widetilde{S}_w$ with $w=sv$ being a reduced expression. 
	 We claim that for any (not necessarily closed) point $x\in \widetilde{S}_w$ the fiber $g^{-1}(x)$ is either isomorphic to $\mathrm{Spec}(\kappa(x))$ or to $\mathbb{P}^1_{\kappa(x)}$: 
	 Indeed, if $g^{-1}(x)$ is $0$-dimensional, then the birational map $g$ becomes a universal homeomorphism of normal varieties around $x$, thus a local isomorphism by Zariski's main theorem; if $g^{-1}(x)$ is $1$-dimensional, then $x$ belongs to $\widetilde{S}_u$ 
	 with $u< v$ and $su<u$, and we can directly see that the fibers of $h\colon S_s \tilde{\times} \widetilde{S}_u\to \widetilde{S}_u$ are projective lines. 
	 Therefore, we have $H^i(g^{-1}(x),\calO_{S_s \tilde{\times}\widetilde{S}_v})=0$ for all $i>0$, which upgrades in the presence of an $F$-splitting to $R^ig_*\calO_{S_s \tilde{\times}\widetilde{S}_v}=0$ for all $i>0$ by \cite[Lemma 1.2.11]{BK07}. Since $\widetilde{S}_w$ is Cohen--Macaulay, Grothendieck--Serre duality yields also $R^if_*\omega_{D_{\dot w}}=\omega_{\widetilde{S}_w}$. This means $\widetilde{S}_w$ has rational singularities, as desired.
	\end{proof}
	
	\begin{remark}\label{remark.simply.connected.reduction}
	The map $G_{\on{sc}}\to G$ from the simply connected group extends to the Iwahori $O$-models, and the induced map on Demazure varieties $D_{\on{sc},\dot w}\to D_{\dot w}$ is an isomorphism, see \cite[Proof of Lemma 3.8]{HR22}. 
	Further, $D_{\on{sc},\dot w}$ factors as a product of Demazure varieties according to the almost simple factors of $G_{\on{sc}}$, and products of (stably) compatibly $F$-split varieties are (stably) compatibly $F$-split \cite[Section 1.3.E (8)]{BK07}.  
	Therefore, in order to verify the assumption of \Cref{lemma.reduction.main.theorem}, we may assume whenever convenient that $G=G_{\on{sc}}$ is simply connected and (by the Weil restriction of scalars case in \cite[Lemma 3.9]{HR22}) {\it absolutely} almost simple and that $\calG=\calI$ is the Iwahori group scheme.
	\end{remark}

	\subsubsection{Picard groups of perfected Schubert varieties}

In this subsection, we calculate the Picard groups of perfected Schubert varieties and the induced map on the Demazure resolution. 
This plays a role later in proving the existence of $F$-splittings for Demazure varieties, which requires the construction of a certain divisor that is more easily done on the Schubert varieties. 

For any $v \in W_{\on{af}}$ we consider the corresponding $(\calI, \calG)$-Schubert variety $S_v$ and its seminormalization $\widetilde{S}_v$.  For the right $W_\calG$-minimal element $w$ above, we fix a choice of reduced expression $\dot w = s_{1} \cdot \ldots \cdot s_{d}$ and consider the Demazure resolution $f\colon D_{\dot w}\to \widetilde S_w$ as in \eqref{equation.reduced.decomposition.Demazure}. 
	For each simple reflection $s\in W_{\on{af}}$ and any choice of isomorphism of $D_{\dot s}\cong \bbP^1_k$, the degree of line bundles induces a well-defined isomorphism $\deg\colon \Pic(D_{\dot s}) \cong\bbZ$.
	
	\begin{lemma}\label{picard.group.demazure.lemma} 
	There is an isomorphism
	\begin{equation}
		\on{Pic}(D_{\dot w}) \overset{\cong}{\longto} \bbZ^d, \;\;\; \calL\mapsto (\deg(\calL|_{D_{\dot s_i}}))_{i=1,\ldots,d}.
		\end{equation}
	\end{lemma}
	\begin{proof}
		The method of \cite[Proposition 3.4]{HZ20} applies as follows.
		Writing $\dot w=\dot s\cdot \dot v$ with $\dot s= s_1$, $\dot v=s_2\cdot\ldots\cdot s_d$ induces an \'etale locally trivial fibration $D_{\dot w}\to D_{\dot s}$ with general fiber $D_{\dot{v}}$.
		The fibration is Zariski locally trivial by \cite[Proposition 8.7 (b)]{PR08}.
		Hence, \cite[Theorem 5]{Mag75} gives an exact sequence $0\to \Pic(D_{\dot s})\to \Pic(D_{\dot w})\to \Pic(D_{\dot v}) \to 0$, which splits by using the section $D_{\dot s}\to D_{\dot w}$.
		The lemma follows by induction.
	\end{proof}
	
	The universal homeomorphism $\widetilde S_{v}\to S_{v}$ induces an isomorphism on perfections \cite[Lemma 3.8]{BS17}, and we denote by $S_{v}^{{\rm{pf}}}$ its common value.
	For each simple reflection $s\in W_{\on{af}} \backslash
        W_\calG$, we have an isomorphism $S_s=\widetilde{S}_s = D_{\dot s} \cong \bbP^1_k$,
        and the degree map uniquely extends to an isomorphism $\deg\colon \Pic(S_s^{{\rm{pf}}})\cong\bbZ[p^{-1}]$ (see \cite[Lemma 3.5]{BS17}); further $\Pic(\widetilde{S}_s) \cong \Pic(S_s^{{\rm{pf}}}) = 0$ if $s \in W_\calG$, since $\widetilde{S}_s \cong S_s \cong \Spec(k)$ in that case.

	\begin{lemma}\label{picard.group.schubert.lemma}
		There is an isomorphism
		\begin{equation}\label{Picard.group.computation}
		\on{Pic}(S_w^{{\rm{pf}}}) \overset{\cong}{\longto} \bigoplus_{s} \bbZ[p^{-1}], \;\;\; \calL\mapsto (\deg(\calL|_{S_s^{{\rm{pf}}}}))_s
		\end{equation}
                where the sum runs over all $s\in\{s_1,\ldots,s_d\}$
                with $s\not\in W_\calG$.  Further, the pullback map
                $\Pic(\widetilde S_w)\to \Pic(S_w^{{\rm{pf}}})$ is
                injective, and its image is a $\bbZ$-lattice.
	\end{lemma}
	\begin{proof}
	The argument in \cite[Proposition 3.9]{HZ20} applied to
        $f\colon D_{\dot w}\to \widetilde S_w$ translates verbatim,
        and we sketch it for the reader's convenience.
	The pullback map $\Pic(\widetilde S_w)\to \Pic(D_{\dot w})$ is injective using the projection formula and the relation $f_*\calO_{D_{\dot w}}=\calO_{\widetilde S_w}$ from \eqref{equation.reduced.decomposition.Demazure}.
	Under the isomorphism $\Pic(D_{\dot w})\cong \bbZ^d$ from
        \Cref{picard.group.demazure.lemma}, in the $i$-th component,
        for all $i=1,\ldots,d$, the map is given by $\calL\mapsto \deg(\calL|_{S_{s_i}})$ if $s_i\not\in W_\calG$ and $\calL\mapsto 0$ else. 
	(Note that if $s \leq w$ and $i_{\dot s}\colon D_{\dot s} \rightarrow D_{\dot w}$ covers $i_s\colon \widetilde{S}_s \hookrightarrow \widetilde{S}_w$, then the map $\Pic(\widetilde{S}_w) \to \Pic(D_{\dot w}) \overset{i^*_{\dot s}}{\rightarrow} \Pic(D_{\dot s})$ factors as $\Pic(\widetilde{S}_w) \overset{i^*_s}{\rightarrow} \Pic(\widetilde{S}_s) \rightarrow \Pic(D_{\dot s})$, hence is zero if $s \in W_\calG$.)
	For any qcqs $\bbF_p$-scheme $X$ one has $\Pic(X^{{\rm{pf}}})=\Pic(X)[p^{-1}]$ by \cite[Lemma 3.5]{BS17}.
	So, passing to perfections implies injectivity of
        \eqref{Picard.group.computation} and that $\Pic(\widetilde
        S_w)$ defines a $\bbZ$-lattice in $\Pic(S_w^{{\rm{pf}}})$.
        
	To prove surjectivity of \eqref{Picard.group.computation}, let $(\la_s)\in \oplus_s\bbZ[p^{-1}]\subset \bbZ[p^{-1}]^d$.
	This induces a line bundle $\calD=\calD(\la_s)$ on $D_{\dot w}^{{\rm{pf}}}$.
It suffices to show that $\calD$ is trivial along the fibers of
$f^{\mathrm{pf}}$, for 
	once we know this $\calD$ descends to $S_{w}^{{\rm{pf}}}$ by $v$-descent for vector bundles on perfect varieties \cite[Theorem 6.13]{BS17}.
	 First, we factor $f^{{\rm{pf}}}$ as partial Demazure resolutions having fibers of dimension at most $1$ as in the proof of \Cref{lemma.reduction.main.theorem}. 
	By induction we may and do replace $f^{{\rm pf}}$ by a corresponding map $g^{{\rm pf}}\colon S^{{\rm pf}}_s \tilde{\times}S^{{\rm pf}}_v \to S^{{\rm pf}}_w$ with $w=sv$ being a reduced expression. 
	Also, by the proof of \Cref{lemma.reduction.main.theorem} the map $g^{{\rm pf}}$ has non-trivial fibers exactly over the union of all $S_u^{\mathrm{pf}}$, with $u\leq v$ such that $su <u$. 
	So, we restrict our attention to the fibers of the map $h^{\mathrm{pf}}\colon S^{{\rm pf}}_s\tilde{\times}S^{{\rm pf}}_u \to S^{{\rm pf}}_u$. 
	By induction on $w$, we know that $\mathrm{Pic}(S_u^{\mathrm{pf}})$ is a free $\mathbb{Z}[1/p]$-module generated on its $L^+\mathcal{I}$-stable projective lines. 
	Following the argument in \Cref{picard.group.demazure.lemma}, the same assertion holds for $\mathrm{Pic}(S^{\mathrm{pf}}_s\tilde{\times}S_u^{\mathrm{pf}})$. 
	In particular, we see that the restriction of $\calD$ to $S_s^{\mathrm{pf}}\tilde{\times}S_u^{\mathrm{pf}}$ is the pullback of some line bundle $\calL$ on $S_u^{\mathrm{pf}}$ along $h^{{\rm pf}}$.
	Hence, it is trivial along the fibers.
	\end{proof}
	
	
	As perfections preserve closed immersions \cite[Lemma 3.4 (viii)]{BS13}, there is a strict $k$-ind-scheme
	\begin{equation}\label{presentation.perfection.flagvariety}
	\Fl_\calG^{{\rm{pf}}}=\on{colim}S_w^{{\rm{pf}}}
	\end{equation}
	lying over $\widetilde\Fl_{\calG}=\on{colim}\widetilde S_w$ from \eqref{normalized.ind.scheme}.
	Their Picard groups are defined as the limit of the Picard groups of the respective Schubert varieties.
	
	\begin{corollary}\label{picard.group.flag.variety.lemma}
	There is an isomorphism
	\begin{equation}\label{Fl_Picard.group.computation}
	\on{Pic}(\Fl_\calG^{0, {\rm{pf}}}) \overset{\cong}{\longto} \bigoplus_s \bbZ[p^{-1}], \;\;\; \calL\mapsto (\deg(\calL|_{S_s^{{\rm{pf}}}}))_s
	\end{equation}
	where $\Fl_\calG^0$ denotes the neutral component and the sum runs over all simple reflections $s\in W_{\on{af}}\backslash W_\calG$.
	The pullback map $\Pic(\widetilde\Fl_{\calG}) \to \Pic(\Fl_\calG^{{\rm{pf}}})$ is injective.
	\end{corollary}
	\begin{proof}
	This is immediate from \Cref{picard.group.schubert.lemma}: For $v\leq w$ in $W_{\on{af}}$ with large enough length, the pullback map $\on{Pic}(S_w^{ {\rm{pf}}})\to \on{Pic}(S_v^{ {\rm{pf}}})$ is an isomorphism, which is the identity map under \eqref{Picard.group.computation}.
	\end{proof}

	\subsubsection{The central charge}
	We assume in this subsection (for simplicity) that $G$ is almost simple and simply connected, compare with \Cref{remark.simply.connected.reduction}.
	In particular, the affine flag variety $\Fl_\calG$ is connected.

	The quotient $L^+\calG\to\calG_k$ induces maps $\Fl_\calG=LG/L^+\calG\to [\Spec\, k/L^+\calG]\to [\Spec\, k/\calG_k]$ to the respective quotient stacks.
	Passing to Picard groups we obtain
	\begin{equation}\label{equivariant.embedding.line.bundles.eq}
	X^*(\calG_k)\cong X^*(L^+\calG) \to \Pic(\Fl_\calG)
	\end{equation}
	Here, the first isomorphism holds because the kernel of $L^+\calG \to \calG_k$ is pro-unipotent.
	The Picard groups of the quotient stacks are the respective character groups because giving a line bundle on such a stack is the same as giving a $1$-dimensional representation of the group, that is, a character.
	
	  \begin{lemma}\label{lemma.equivariant.lb}
    The group $\mathrm{Pic}([LG\backslash \Fl_{\calG}])$ of isomorphism classes of line
      bundles on $\Fl_{\calG}$ equipped with an $LG$-equivariant structure naturally identifies with $X^*(\calG_k)$ via
      the induction map
	\begin{equation}\label{eqn.equiv.constr.ker.c}\mu \mapsto \calL(\mu):=LG \times^{L^+\calG}\calO_\mu,
	\end{equation}
	where $\calO_{\mu}$ is the equivariant line bundle on $L^+\calG$ attached to $\mu$.
	\end{lemma}
	\begin{proof}
          This is immediate from the isomorphisms in
          \eqref{equivariant.embedding.line.bundles.eq}: the inverse
          to the induction map is given by pullback of
          $LG$-equivariant line bundles to the origin of
          $\Fl_{\calG}$, noticing that they carry an action of
          $L^+\calG$, that is,
          $[LG\backslash \Fl_\calG]=[\Spec\, k/L^+\calG]$ in terms of
          (\'etale) stacks.
	\end{proof}
	
	We now pass to perfections in order to make the map
        \eqref{equivariant.embedding.line.bundles.eq} explicit,
        compare \Cref{picard.group.flag.variety.lemma}.  So choosing
        any presentation of $LG$ by affine schemes, we denote by
        $LG^\pf$ the colimit of the perfections of the constituents.
        As $k$ is perfect, we can equivalently use the relative
        Frobenius over $k$ to form $LG^\pf$, so it is naturally an
        ind-affine $k$-group ind-scheme.
	
	After perfection, we deduce from \Cref{lemma.equivariant.lb} and \Cref{equivariant.embedding.line.bundles.eq} the homomorphism
	\begin{equation}\label{equation.characters.induced}
	X^*(\calG_k)[p^{-1}] \to \mathrm{Pic}(\Fl_{\calG}^\pf),
	\end{equation} 
	whose image identifies with the line bundles admitting an $LG^\pf$-equivariant structure. 
	In order to explicitly describe \eqref{equation.characters.induced}, we fix the standard basis $\epsilon_i=(0,\dots, 1,\dots, 0)$ of $\mathrm{Pic}(
        \Fl_{\calI}^{\rm{pf}})\cong \oplus_s \bbZ[p^{-1}]$ (see \Cref{picard.group.flag.variety.lemma} for $\calG=\calI$ being the Iwahori).
         It will be convenient for us to fix a certain enumeration of the simple reflections.
   \begin{lemma}\label{lemma.affine.reflection}
   	There exists a simple reflection $s_0$ such that the unique standard maximal parahoric $\calG_0$ with $s_0 \notin W_{\calG_0}$ satisfies the following: the reductive quotient of the special fiber $\calG_{0,k}$ is simply connected and its root system equals the non-multipliable roots of $\Phi_G$.
   \end{lemma}     
\begin{proof}
 For any positive simple
 affine root $\alpha_s$ in the affine root system $\Psi_G$ in the
 sense of \cite[Definition 4.3.4]{KP23} associated with a simple reflection $s$, let $a_s \in \Phi_G$ be the
 gradient of $\alpha_s$. For any enumeration $s_0, \dots, s_n$ of the simple reflections, we write $a_i:=a_{s_i}$ for $i=0,\dots, n$. We claim that there exists a choice of enumeration such that the $a_i$ for $i>0$ form a basis of $\Phi_G^{\mathrm{nm}}$, the
 sub-root system of non-multipliable roots. In order to see that this is possible, we consider the folowing cases: either $\Phi_G$ is reduced,
 and this amounts to the choice of a special vertex in the fundamental alcove not fixed by $s_0$; or $\Phi_G$ is not
 reduced, and we need to ensure the existence of special vertices which
 are not extra special in the sense of \cite[Proposition
 1.5.39]{KP23}, which can be verified in \cite[Table
 1.5.51]{KP23}. From now on, we fix such an enumeration and claim that the standard maximal parahoric $\calG_0$ attached to $s_0$ satisfies the conditions in the lemma.
 
 In what follows, we canonically identify
 the character and cocharacter groups of the $K$-split torus $S$ with
 those of the special fiber $\calS_k$ of its connected N\'eron
 $O$-model $\calS$.  Note that $\calS_k$ defines a maximal split torus
 of the reductive quotient of $\calG_{0,k}$, because $k$ is
 algebraically closed.  By construction, the $a_i$ for $i>0$ define
 roots of the reductive quotient of $\calG_{0,k}$. In particular, the
 coroots $a_i^\vee$ for $i > 0$
  (which are
  non-divisible) form a basis of the dual root system $\Phi_G^\vee$
  and hence of $X_*(S)$ by our assumption on $G$.
\end{proof}   
    From now on, we fix an enumeration $s_0, s_1, \dots,
        s_n$ of the simple reflections with $s_0$ being as in \Cref{lemma.affine.reflection}. 
  With our numbering system above in terms of our choice of
  special vertex, this has the following explicit description: for
  $i>0$, $a_i$ is the non-multipliable relative root whose reflection
  is $s_i$, and $a_0$ is the negative of the highest multipliable
  relative root.

        	\begin{lemma}\label{lemma.kac.moody.coeff.inj}
		Let $a_i^\vee\in X_*(S)$ be the coroot associated to the root $a_i$ as defined above. Under the isomorphism \eqref{Fl_Picard.group.computation}, the map
		\eqref{equation.characters.induced} is given by
		\begin{equation}\label{lemma.kac.moody.coeff.eq}
		\mu \mapsto \sum \langle a_i^\vee, \mu\rangle \epsilon_i, 
		\end{equation}
		where the sum runs over all $i=0,\ldots,n$ with $s_i\not\in W_\calG$.
		Thus, \eqref{equation.characters.induced} is injective
                and has cokernel free over $\bbZ[p^{-1}]$ of rank $1$.
	\end{lemma}

\begin{proof}
  Let $ \calP_i$ be the minimal standard parahoric such that
  $L^+\calP_i/L^+\calI = S_{s_i}$. The reductive quotient of the
  special fiber of $\calP_i$ has simply connected cover isomorphic to
  $\on{SL}_2$ with positive coroot $a_i^\vee$.  Therefore, the pullback to
  $S_{s_i}^{\rm{pf}}$ of the equivariant line bundle $\calL(\mu)$ attached to a weight $\mu \in X^*(S)[p^{-1}]\cong X^*(\calI_k)[p^{-1}]$ is
  isomorphic to $\calO(\langle a_i^\vee, \mu\rangle)$, hence
  \eqref{lemma.kac.moody.coeff.eq} holds. It is well-known from the theory of algebraic groups that $X^*(\calG_k)$ is a
  direct summand of $X^*(S)$, compare with \cite[Corollary A.2.7]{CGP15}. Hence, to deduce injectivity of \eqref{equation.characters.induced} and
  freeness of its cokernel, we may and do assume that $\calG=\calI$ is the Iwahori. Due to the fact that $\calS_k$ identifies with a maximal torus in the reductive quotient of $\calG_{0,k}$, which is simply connected with roots $\Phi^{\rm nm}_G$, the coroots $a_i^\vee$ for $i>0$ form a basis of $X_*(S)$. So its dual basis $\omega_i$ form a basis of $X^*(S)$, and thus \eqref{lemma.kac.moody.coeff.eq} admits a section.  
  Finally, to see that the cokernel has rank $1$ for arbitrary $\calG$, we proceed as follows. First, we notice that for any $i=0, \dots, n$, the set $a_j$ for $j\neq i$ forms a basis of $X^*(S)_\bbQ$, because otherwise the affine reflections $s_j$ for $j \neq i$ would have a positive-dimensional intersection in $\scrA(G,S)$. Suppose $W_\calG$ contains exactly $m<n+1$ many simple reflections and notice that the associated relative roots are still linearly independent in $X_*(S)_\bbQ$ by our previous observation. Let $\calS_k^{\der}$ denote the maximal torus of the derived subgroup of $\calG_k^{\mathrm{red}}$ and notice that $X^*(\calS_k^\der)$ has rank $m$. We deduce that the cokernel $X^*(\calG_k)_\bbQ$ of $X^*(\calS_k^\der)_\bbQ \to X^*(S)_\bbQ$ has rank $n-m$, whereas $\mathrm{Pic}(\Fl_{\calG}^\pf)$ has rank $n+1-m$ by \Cref{Fl_Picard.group.computation}.
      \end{proof}

	Using that \eqref{equation.characters.induced} is injective
        and has free cokernel of rank $1$ (see
        \Cref{lemma.kac.moody.coeff.inj}), we construct a  homomorphism	
	\begin{equation}\label{equation.central.charge}
	 \on{Pic}(\Fl_{\calG}^{\pf})\rightarrow \bbZ[p^{-1}], \; \calL\mapsto c_{\calL},
	\end{equation}
	called the \textit{central charge homomorphism}, uniquely
        characterized by the following properties: its kernel is
        $X^*(\calG_k)[p^{-1}]$; it factors through
        $\on{Pic}(\Fl_{\calG}^{\pf})\to \on{Pic}(\Fl_{\calI}^{\pf})$;
        for $\calG=\calI$ the standard $\bbZ$-lattice
        $\oplus_s\bbZ\subset \oplus_s\bbZ[p^{-1}]\cong
        \Pic(\Fl_\calI^\pf)$ (see
        \Cref{picard.group.flag.variety.lemma}) maps onto
        $\bbZ\subset \bbZ[p^{-1}]$, preserving positive degrees. Please note that this map is just the $\mathbb{Z}[p^{-1}]$-linearization of the usual central charge as defined in \cite[Equation (2.2.3)]{Zhu14}. The only reason we defined it in the perfect setting is because we have not yet proved \Cref{theorem_frobenius_split_allp}, so we do not control the Picard group of the $\Fl_{\calG}$, but only of its perfection. We remark that the homomorphism \eqref{equation.central.charge} is surjective when
        $\calG=\calI$ is the Iwahori, but usually not for general
        parahorics, see \cite[Section 2.2, page 12]{Zhu14}.

	\begin{lemma}\label{lemma.kac.moody.coeff}
		Let $\omega_i \in X^*(S)$ for $i=1 ,\ldots,n$ be the dual basis to $a_i^\vee$.
		Under the isomorphism \eqref{Fl_Picard.group.computation}, the map
		\eqref{equation.central.charge} is given by
		\begin{equation}
		(\la_i) \mapsto \la_0-\sum_{i>0} \langle a_0^\vee, \omega_{i} \rangle \la_i,
		\end{equation}
                 where we use the convention that $\lambda_i=0$ whenever $s_i \in W_{\calG}$.
		In particular, the coefficients $1$ and $-\langle a_0^\vee,\omega_{i}\rangle$ are the numbers attached in \cite[Section 6.1]{Kac90} to the vertices of the dual affine Dynkin diagram of $G$.
	\end{lemma}

\begin{proof}
	The proof of \Cref{lemma.kac.moody.coeff.inj} shows that $\calL(\omega_i)$ is the image of
        $\langle a_0^\vee, \omega_i\rangle \epsilon_0+\epsilon_i$
        under the bijection \eqref{Fl_Picard.group.computation}.
        Hence, we get $c(\epsilon_i)=-\langle a_0^\vee, \omega_i\rangle c(\epsilon_0)$. 
        So $c(\oplus_s \bbZ) \subset \bbZ c(\epsilon_0)$ and by our choice of normalization $c(\epsilon_0)=1$, thus $c(\epsilon_i)=-\langle a_0^\vee, \omega_i\rangle $ for $i>0$.
	
	Finally, for the comparison with Kac--Moody theory, this can be seen by inspecting \cite[Theorem 4.8, Tables Aff 1-3]{Kac90} or the construction of the central charge for untwisted and twisted Kac--Moody algebras, see \cite[Theorems 7.4 and 8.3]{Kac90}. Alternatively, we may observe that these coefficients are combinatorial data that do not really depend on the arithmetic properties of $G$, so we may assume the latter to be tamely ramified, in which case $\Fl_{\calG}$ identifies with a Kac--Moody flag variety, see \cite[9.h and Proposition 10.1]{PR08} and also \cite[Annexe A]{Lou23}.
\end{proof}

Recall that for $\calG={\rm{GL}}_n$ we have an ample line bundle $\calL_{\on{det}}=\calO(1)$ on $\Fl_{{\rm{GL}}_n}$ such that $c(\calL_{\on{det}})=1$. Pulling it back along the adjoint representation $\mathrm{ad}\colon \Fl_{\calG} \to \Fl_{\mathrm{GL}(\mathrm{Lie}\calG)}$, we get an ample line bundle $\calL_{\mathrm{ad}}$ on $\Fl_{\calG}$ whose central charge can still be determined:

\begin{lemma}\label{lem.central.charge.adjoint}
	The central charge $c(\calL_{\mathrm{ad}})$ of the adjoint line bundle is equal to $2h^\vee$, where $h^\vee$ is the dual Coxeter number of the split form of $G$.
\end{lemma}

\begin{proof}
	We invoke \cite[Lemma 4.2]{Zhu14} at Iwahori level, which shows that $\calL_{\mathrm{ad}}$ has degree $2$ when restricted to every $S_{s_i}$, and which does not use any tameness assumptions. But it is well-known that the sum $1-\sum \langle \omega_i,a_0^\vee \rangle $ equals the dual Coxeter number. For general parahoric level, there is a reduction step in the remaining paragraphs of the proof of \cite[Proposition 4.1]{Zhu14} that follow the Iwahori lemma cited above.
\end{proof}

	A key property of \eqref{equation.central.charge} is its constancy along the fibers of Beilinson--Drinfeld Grassmannians, and we extend the results \cite[Lemma 18, Remark 19]{He10} and \cite[Proposition 4.1, Corollary 4.3]{Zhu14} from tamely ramified groups to general reductive groups as follows. 
	Let $\Gr_\calG\to \Spec(O)$ be the Beilinson--Drinfeld Grassmannian, see \cite[Definition 2.3]{Ric16} and \cite[Section 0.3]{Ric19b} for a definition independent of auxiliary choices. 
	Then $\Gr_\calG\to \Spec(O)$ is an ind-projective ind-scheme, its generic fiber $\Gr_{\calG,K}$ is equivariantly isomorphic to the affine Grassmannian $\Gr_G$ formed using an additional formal parameter \cite[Section 0.2]{Ric19b}, whereas its special fiber $\Gr_{\calG,k}$ is equal to $\Fl_\calG$.
	Looking ahead to the proof of \Cref{line.bundle.extension} below, we note that the line bundle $\calL_{\mathrm{ad}}$ above extends to a line bundle on $\Gr_{\calG }$, by the same construction (use \cite[$\S2.5$]{Ric16}); we denote the extension also by $\calL_{\mathrm{ad}}$.
	By our assumptions on the group, we can write $G=\Res_{L/K}G_0$ for some finite, separable field extension $L/K$ and some absolutely almost simple, simply connected reductive $L$-group $G_0$.
	 Given a scheme $X$, let $\mathrm{Pic}(X)_\bbQ$ denote the rationalized Picard group of $X$. For an ind-scheme $X$, we define $\mathrm{Pic}(X)_\bbQ$ as the limit of the $\mathrm{Pic}(X_i)_\bbQ$ along a presentation (in all cases considered in this paper, this will match the $\bbQ$-localization of $\mathrm{Pic}(X)$).

	\begin{lemma}\label{line.bundle.extension}
		 The following properties hold:
		\begin{enumerate}
			\item 
			\label{line.bundle.extension.1}
			The map $\on{Pic}(\Gr_{\calG})_\bbQ \to \on{Pic}(\Fl_{\calG})_\bbQ$ is surjective.
			\item 
			\label{line.bundle.extension.2}
			Every $\calL \in \mathrm{Pic}(\Gr_{\calG})_\bbQ$ has geometric generic fiber isomorphic to $\calO(c_{ \calL_k})$, the $c_{ \calL_k}$-th tensor power of $\boxtimes_{[L:K]}\calO(1)$ on $\Gr_{G,\bar K}\cong \prod_{[L:K]}\Gr_{G_0,\bar K}$. 
		\end{enumerate}
	\end{lemma}

\begin{proof}
	There is a natural map $\Gr_{\calG} \to [\Spec \,k /\calG_{k}]$ to the classifying stack of $\calG_k$-bundles over $k$, given by forgetting the modification and then restricting the torsor to the subscheme defined by the principal ideal $t$.
	This map factors the map $\Fl_\calG \to [\Spec \,k /\calG_{k}]$ (compare \Cref{equivariant.embedding.line.bundles.eq}) under the identification $\Gr_{\calG,k}=\Fl_\calG$.
	Passing to Picard groups, we get maps $X^*(\calG_{k}) \to \on{Pic}(\Gr_{\calG})\to \on{Pic}(\Fl_{\calG})$ whose composition is \eqref{equivariant.embedding.line.bundles.eq}.
	After rationalizations, the maps are injective. 
	Further, $\calL_\ad$ and $\mathrm{ker}(c)_\bbQ=X^*(\calG_{k})_\bbQ$ generate the $\bbQ$-vector space $\on{Pic}(\Fl_{\calG})_\bbQ$ by \Cref{picard.group.flag.variety.lemma}.
	This shows \eqref{line.bundle.extension.1}.

	For \eqref{line.bundle.extension.2}, we start by noticing that its conclusion is satisfied by the image of $X^*(\calG_{k})_\bbQ \to \on{Pic}(\Gr_{\calG})_\bbQ$. Indeed, the map $\Gr_{G,\bar K} \to [\Spec \,k /\calG_{k}]$ factors through the trivial torsor by Beauville--Laszlo gluing, so it must induce the zero map on the rationalized Picard group. Moreover, the conclusion holds as well for $\calL_{\mathrm{ad}}$ defined over $\Gr_{\calG}$ again by pulling back $\calL_{\mathrm{det}}$ along the adjoint map to the Lie algebra. Indeed, on the geometric generic fiber $\Gr_{G,\bar K}\cong \prod_{[L:K]} \Gr_{G_0,\bar K}$, the line bundle $\calL_{\mathrm{ad}}$ becomes isomorphic to $\calO(2h^\vee)$, where $h^\vee$ is the dual Coxeter number of $G_0$, by \Cref{lem.central.charge.adjoint} applied to each of the factors $G_0$. On the special fiber $\Fl_{\calG}$, we also know by \Cref{lem.central.charge.adjoint} that $c_{\calL_{\mathrm{ad}}}=2h^\vee$. 

        Since the previous explicitly given rationalized line bundles
        generate $ \on{Pic}(\Fl_{\calG})_\bbQ$ as seen already, we may
        and do assume that our abstract rationalized line bundle
          $\calL$ on $\Gr_{\calG}$ has trivial special fiber
        $\calL_{k}=\calO$.  Let $\mu$ be a conjugacy class of
        cocharacters in $G_{\bar K}$ with reflex field $E\supset K$.
        Let $M_{\calG,\mu}$ be the orbit closure of $S_{G,\mu}$ over
        $ O_E$, see \Cref{loc_mod_def}, and suppose that $\mu$ is
        supported on exactly one almost simple factor of $G_{\bar
          K}$. Then, $\mathrm{Pic}(S_{G,\mu, \bar K})_\bbQ$ is
        $1$-dimensional by \Cref{picard.group.schubert.lemma}. Assume
        for the sake of contradiction that $\calL_{\bar K}$ is
        anti-ample on $S_{G,\mu,\bar K}$ (if not, take its
        inverse). It is therefore equal to the restriction of
        $\calL_{\mathrm{ad},\bar K}^{-q}$ for some $q \in
        \bbQ_{>0}$. Replacing $\calL$ by its product with
        $\calL_\ad^q$, we may now ensure that $\calL_k$ is ample and
        $\calL_{\bar K}$ is trivial on $\Gr_{G,\mu,\bar K}$. This
        contradicts openness of the ample locus of $\calL$ on
        $ M_{\calG,\mu}$, see \cite[Corollaire 9.6.4]{EGAIV3}. In
        particular, we conclude that $\calL_{\bar K}$ must be trivial
        on $S_{G,\mu, \bar K}$. Letting $\mu$ run over all coweights
        with irreducible support, we deduce from
        \Cref{picard.group.flag.variety.lemma} that $\calL_{\bar K}$
        is trivial.
\end{proof}

	Suppose we are given a map $f\colon \calG_1 \to \calG_2$ of
        parahoric $O$-models of simply connected, almost simple $
        K$-groups $G_1$ and $G_2$. 
	We have an induced pull-back map \begin{equation}f^*\colon \mathrm{Pic}(\Fl_{\calG_2}^{\mathrm{pf}})\to \mathrm{Pic}(\Fl_{\calG_1}^{\mathrm{pf}})
	\end{equation}
	that sends equivariant line bundles with respect to $LG_2^{\mathrm{pf}}$ to those with respect to $LG_1^{\mathrm{pf}}$. In particular, we get a homomorphism of cokernels defined by their central charges and it follows that $c_1(f^*\calL)=d(f)c_2(\calL)$ where $d(f) \in \bbZ_{\geq 0}$ is independent of $\calL$ and $c_i$ denote the central charges of the respective Picard groups. 
	Here, the non-negativity of $d(f)$ holds because pullback preserves semi-ampleness, and $d(f)$ is an integer because the map of Picard groups also exists on the non-perfected affine flag varieties.
	
	From the constancy of the central charge we draw the following consequence:
	
	\begin{corollary}\label{cor.charge.deg.res.sc}
		Let $L/K$ be a finite separable extension and consider the natural map \begin{equation}f\colon \calG \to \Res_{O_L/O_K}(\tilde{\calG}),\end{equation} extending the unit of adjunction for $\Res_{L/K}$, where $\tilde{\calG}$ is the associated parahoric $O_L$-model of $G_L$ induced by the map $\scrB(G,K)\to \scrB(G,L)$. Then $d(f)=[L:K]$.
	\end{corollary}

\begin{proof}
	Thanks to \Cref{line.bundle.extension}, we can read off the integer $d(f)$ from the map of affine Grassmannians $\Gr_G \to \Gr_{\Res_{L/K}G_L}$ after base changing to $\bar K$. But then $\Res_{L/K}G_L$ splits over $\bar K$ as a product of $[L:K]$-many copies of $G_{\bar K}$, so $\calO(1)= \boxtimes_{[L:K]} \calO(1)$ pulls back to $\calO([L:K])$ as desired.
\end{proof}
	\subsubsection{The Demazure variety is stably compatibly $F$-split}\label{subsection_demazure}
	 In order to finish the proof of \Cref{theorem_frobenius_split_allp}, it remains to show that the assumption of \Cref{lemma.reduction.main.theorem} holds, that is, the Demazure variety $D_{\dot w}$ is stably compatibly $F$-split with $D_{\dot v}$ for all $\dot v$ of colength $1$ in $\dot w$. 
	 By \Cref{remark.simply.connected.reduction}, we may and do assume that $G$ is simply connected, absolutely almost simple and that $\calG=\calI$ is the Iwahori group scheme.
		As in \cite[Section 8]{PR08} (for proving $F$-splitness) and \cite[Section
		5]{Cas22} (for proving stable $F$-splitness), we aim to apply the
		Mehta--Ramanathan splitting criterion, see
		\cite[Theorem 5.3.1]{BS13} and \cite[Proposition 1.3.11]{BK07}, to
		$D_{\dot{w}}$ together with its divisors
		$D_{\dot{v}}$. 
		We need the following result for this.
		
		\begin{lemma}\label{lem.crit.line.bundle}
			There exists a unique line bundle $\calL_{\mathrm{crit}}\in\mathrm{Pic}(\widetilde{\Fl}_{\calG})$ such that $\calL_{\mathrm{crit}}^{\otimes 2}\simeq \calL_{\mathrm{ad}}$. 
		\end{lemma}
		
		\begin{proof}
			We first prove uniqueness.
			Note that $f^*\colon \mathrm{Pic}(\widetilde{S}_w)\to \mathrm{Pic}(D_{\dot w})$ is injective as $f_*f^*\calL=\calL$ by the projection formula and by $f_*\calO_{D_{\dot w}}=\calO_{\widetilde S_w}$, see \eqref{equation.reduced.decomposition.Demazure} for the latter. 
			By \Cref{picard.group.demazure.lemma}, we see that $\mathrm{Pic}(\widetilde{S}_w)$ is torsion-free for all $w\in W_{\aff}$ and so is $\mathrm{Pic}(\widetilde{\Fl}_{\calG})$. 
			Hence, the map $\pi^*\colon \mathrm{Pic}(\widetilde{\Fl}_{\calG})\to \mathrm{Pic}(\Fl_{\calG}^{\mathrm{pf}})$ is injective. 
			In particular, $\calL_{\mathrm{ad}}$ admits at most one square root.
			
			Next, we prove existence. 
		Recall that $\calL_{\mathrm{ad}}$ restricts to $\calO(2)$ on
			every $S_{s}$, so it admits a square
			root inside $ 
			\mathrm{Pic}(\Fl_{\calG}^{\mathrm{pf}})$ of central
			charge equal to $h^\vee$, see
			\Cref{picard.group.flag.variety.lemma} and
			\Cref{lem.central.charge.adjoint}. 
			There are inclusions of $\mathbb{Z}$-lattices
			\begin{equation}\mathrm{Pic}(\widetilde{\Fl}_{\calG})\subset\oplus_s\mathbb{Z}\subset \oplus_s \mathbb{Z}[p^{-1}]=\mathrm{Pic}(\Fl_{\calG}^{\mathrm{pf}}).\end{equation} 
			(Note that the lattices are in fact equal, which only follows after finishing the proof of \Cref{theorem_frobenius_split_allp}.)
			The cokernel of the inclusion is $p$-power torsion. 
			If there were no square root $\calL_{\mathrm{crit}}$ on $\widetilde{\Fl}_{\calG}$, then the element $(1,\dots,1) \in \oplus_s \bbZ$ would be a non-trivial $2$-torsion point of the cokernel $\mathrm{Pic}(\Fl_{\calG}^{\mathrm{pf}})/\mathrm{Pic}(\widetilde{\Fl}_{\calG})$, which yields a contradiction unless $p=2$. 
			So, $\calL_{\mathrm{crit}}$ exists for whenever $p>2$.
			
			Now, let $p=2$.
			Informally speaking, we aim to show that $\mathrm{Pic}(\widetilde{\Fl}_{\calG})$ is large enough as follows. 
			\Cref{lemma.equivariant.lb} and \Cref{lemma.kac.moody.coeff.inj} implies that $\ker(c) \cap \oplus_s \bbZ$ already lies in $\on{Pic}(\widetilde{\Fl}_{\calG})$. 
			Since $c(1,\dots,1)=h^\vee$, it is enough to prove the inclusion
			\begin{equation} \label{c-contain}
				c(\on{Pic}(\widetilde{\Fl}_{\calG})) \supset h^\vee\bbZ,
			\end{equation}
			where we recall the normalization of $c$ from \eqref{equation.central.charge}.
			In order to verify \eqref{c-contain},
			let $e \leq 3$ denote the degree of the smallest extension $L/K$ whose Galois hull $\widetilde{L}/K$ splits $G$. The flag variety of the corresponding Iwahori model $\widetilde{\calG}$ over $O_{\widetilde{L}}$ admits a line bundle with central charge $1$ by \cite[Theorem 7]{Fal03}. By \Cref{cor.charge.deg.res.sc}, we obtain the inclusion $e!\bbZ \subset c(\on{Pic}(\widetilde{\Fl}_{\calG}))$. 
			Looking at the classification of \cite[Planches]{Bou68}, we see that $e!$ always divides $h^\vee$, unless $G=\on{SU}_{2n+1}$ is an odd-dimensional unitary group. 
			
			Finally, if $G=\on{SU}(V,q)$ is a unitary group, where $V$ is a $L$-vector space and $q\colon V \rightarrow L$ is a semi-regular $L$-hermitian form, we follow the implicit argument that had already been covered in \cite[Lemma 8.3]{Zhu14} for $p>2$, but now for all primes. Namely, in \Cref{lem_unitary_inside_orthogonal} below, we will consider the natural map of $K$-groups
			\begin{equation}
				\on{SU}(V,q) \rightarrow \on{SL}(_KV)
			\end{equation}
			where $_KV$ is
			$V$ regarded as a $K$-vector space, and construct a certain non-degenerate quadratic
			form $r\colon {}_KV
			\rightarrow K$ such that the above map factors through $\on{SO}(_KV,r)$. Notice this
			solves our problem of constructing a
			line bundle $\calL$ satisfying $c(\calL)=1$, since the determinant has a square root given by the Pfaffian, see \cite[Section 4.2]{BD91} and especially \cite[Section 4.2.16]{BD91} when $p=2$.
			\end{proof}
			
			The following lemma is used towards the end in the proof of \Cref{lem.crit.line.bundle}.
			
			\begin{lemma}\label{lem_unitary_inside_orthogonal}
				Let $L/K$ be a quadratic extension, $V$ a $L$-vector space and $q\colon V \rightarrow L$ a semi-regular $L$-hermitian form.
				There is a non-degenerate quadratic form $r\colon {}_KV \rightarrow K$ such that $\on{SU}(V,q)$ lies inside $\on{SO}({}_KV,r)$.
			\end{lemma}
			
			\begin{proof}
				If $p>2$, this is a well-known result in the theory of $L$-sesquilinear and $K$-bilinear forms, see \cite[Section 1.2.2]{PR09}, so from now on we assume $p=2$.
				
				Decomposing into orthogonal summands, we may assume either \begin{equation}(V,q)=(L,x \mapsto N(x))\end{equation} is one-dimensional semi-regular,  or \begin{equation}(V,q)=(L^2,(x,y)\mapsto x\sigma(y)+\sigma(x)y)\end{equation} is a two-dimensional regular hermitian hyperbolic plane. 
				
				In the first case, taking \begin{equation}r=\on{tr}(\lambda q)\colon (x_1,x_2) \mapsto x_1^2+x_1x_2+N(\lambda)x_2^2, \end{equation} where $\on{tr}(\lambda)=1$, gives us a regular symmetric $K$-hyperbolic plane, as $1-4N(\lambda)=1 \neq 0$ as $p=2$. As for the second case, the quadratic form \begin{equation}r=\on{tr}(\lambda q)\colon(x_1,x_2,y_1,y_2) \mapsto x_2y_1 +x_1y_2\end{equation} clearly decomposes into the orthogonal sum of two regular symmetric $K$-hyperbolic planes.
 \end{proof}
		
\begin{remark}
  The construction of $\calL_{\textup{crit}}$ on the
  \textit{seminormalized} affine flag variety is used in order to
  apply the Mehta--Ramanathan criterion.  It would be interesting to
  find a uniform proof for all $G$ and $p$.  Recall that \cite[Theorem
  7]{Fal03} provides a construction for split $G$, which is extended
  in \cite[Corollary 4.3.10]{Lou23} for tame $G$, using negative loops
  groups that seem, however, not to exist for wildly ramified $G$.
  Also, the work \cite{PR08} refers to a construction in
  \cite[Proposition 3.19]{Gor01} for $G=\mathrm{GL}_n$, which we were
  not able to generalize to other groups.
\end{remark}
			
	Now, we are ready to finish the proof of \Cref{theorem_frobenius_split_allp}. Let $f\colon D_{\dot w}\to \widetilde S_w$ be the Demazure resolution, compare \eqref{equation.reduced.decomposition.Demazure}.
		The anti-canonical line bundle admits the formula
		\begin{equation}\label{eq.finish.proof}
		\omega^{-1}_{D_{\dot{w}}}=\calO(\partial D_{\dot{w}})\otimes f^*\calL_{\on{crit}}
		\end{equation}
		by the argument of \cite[Proposition 2.2.2]{BK07}, and the fact that $\calL_{\on{crit}}$ has degree $1$ on every projective line $S_s$.
		To apply the Mehta--Ramanathan
		criterion, see \cite[Theorem 5.3.1]{BS13} and
		\cite[Proof of Theorem 5.8]{Cas22}, we must produce a
		section of the $(q-1)$-th power of
		$\calL_{\mathrm{crit}}$ (for some power $q$ of $p$) avoiding the origin (i.e., the intersection of all the divisors $D_{\dot v}$). Note that $\calL_{\mathrm{crit}}$ is an ample line bundle, because so is its square $\calL_{\mathrm{ad}}$. We deduce that any sufficiently large power of $\calL_{\mathrm{crit}}$ is very ample on $\widetilde{S}_w$, and therefore
		$f^*\calL^{q-1}_{\on{crit}}$ will be basepoint free for some sufficiently large power $q\gg 0 $ of $p$.  
	
	\subsubsection{Picard groups of seminormalized Schubert varieties}
	
	Using the already proven \Cref{theorem_frobenius_split_allp}, we can actually upgrade the previous results on Picard groups to seminormalized Schubert varieties.
	
	\begin{lemma}\label{picard.group.schubert.lemma.normal}
		There is an isomorphism
		\begin{equation}\label{Picard.group.computation.normal}
		\on{Pic}(\widetilde{S}_w) \overset{\cong}{\longto} \bigoplus_{s} \bbZ, \;\;\; \calL\mapsto (\deg(\calL|_{S_s}))_s
		\end{equation}
		where the sum runs over all $s\in\{s_1,\ldots,s_d\}$ with $s\not\in W_\calG$.
	\end{lemma} 
	\begin{proof}
		Recall the notation $f\colon D_{\dot w}\to \widetilde S_w$ for the Demazure resolution from \eqref{equation.reduced.decomposition.Demazure} and the computation of $\Pic(D_{\dot w})$ from \Cref{picard.group.demazure.lemma}.
		As explained in \Cref{picard.group.schubert.lemma}, the pullback map $\Pic(\widetilde S_w)\to \bigoplus_{s} \bbZ$ is injective.
		For surjectivity, let $(\lambda_s)\in \oplus_s\bbZ$ and denote by $\calD=\calD(\lambda_s)$ the corresponding line bundle on $D_{\dot w}$.
		
		We show that $\calL:=f_*\calD$ is a line bundle, and that the canonical map $f^*\calL\to \calD$ is an isomorphism.
			As in the proof of \Cref{lemma.reduction.main.theorem} we factor $f$ into successive partial Demazure resolutions, each having fibers of dimension at most $1$.
	 By induction we replace $f$ by one of those maps $g\colon S_s \tilde{\times}\widetilde{S}_v \to \widetilde{S}_w$.
		 By the proof of \Cref{picard.group.schubert.lemma},
		we already know that the restriction of $\calD$ to the fibers of $g$ is trivial after passing to perfections. By the proof of \Cref{lemma.reduction.main.theorem}, we know that the fibers of $g$ are either $\mathrm{Spec}(\kappa(x))$ or $\mathbb{P}^1_{\kappa(x)}$, so their Picard groups are torsion-free and $\calD$ has trivial restriction to all fibers of $g$. 
		By \Cref{theorem_frobenius_split_allp}, our varieties have rational singularities\footnote{Strictly speaking, \Cref{theorem_frobenius_split_allp} only refers to the $\widetilde{S}_w$ and not their partial Demazure resolutions, but the proof given in the previous section proceeds by descent from $D_{\dot{w}}$, so those also have rational singularities.}, so \cite[Theorem 12.1 (i)]{Lip69} applies to show that $\calD$ is Zariski locally trivial on the base.
		Using rational singularities again shows $g_*\calD$ is a line bundle, and that $g^*\calL\to \calD$ is an isomorphism.
              \end{proof}

	\begin{corollary}\label{picard.group.flag.variety.lemma.normal}
		There is an isomorphism
		\begin{equation}\label{Fl_Picard.group.computation.normal}
		\on{Pic}(\widetilde{\Fl}_\calG^{0}) \overset{\cong}{\longto} \bigoplus_s \bbZ, \;\;\; \calL\mapsto (\deg(\calL|_{S_s}))_s
		\end{equation}
		where $\widetilde{\Fl}_\calG^0$ denotes the neutral component and the sum runs over all simple reflections $s\in W_{\on{af}}\backslash W_\calG$.
	\end{corollary}
	\begin{proof}
		This is immediate from
		\Cref{picard.group.schubert.lemma.normal}, as
		$\on{Pic}(\widetilde{S}_w)$ is again independent of
		$w$ for sufficiently large lengths by \eqref{Picard.group.computation.normal}.
	\end{proof}

	\Cref{picard.group.schubert.lemma.normal} admits the following slight generalization (see \Cref{picard.group.schubert.lemma.normal.union}) which is used in \Cref{section_lm}.
	We first need an elementary lemma:

	\begin{lemma}\label{intersections.seminormalized.Schubert.varieties}
	Finite unions of seminormalized Schubert varieties in $\widetilde \Fl_\calG$ are seminormal and stable under finite intersections. 
	\end{lemma}
	\begin{proof}
	 Due to the compatible $F$-splitting of seminormalized Schubert varieties from \Cref{theorem_frobenius_split_allp}, their finite union (and, finite intersection) is again $F$-split, hence $F$-injective (and reduced) and therefore seminormal by \cite[Theorem 4.7]{ Sch09}. In particular, if $S_{w_1},\ldots, S_{w_n}\subset \Fl_\calG$ are Schubert varieties, then the maps $\cup_{i=1}^n \widetilde S_{w_i}\to \cup_{i=1}^n S_{w_i}$ and $\cap_{i=1}^n \widetilde S_{w_i}\to \cap_{i=1}^n S_{w_i}$ are universal homeomorphisms and induce isomorphisms on all residue fields, and so identify the respective sources as the seminormalizations of their targets. The lemma follows.
	\end{proof}

	\begin{proposition}\label{picard.group.schubert.lemma.normal.union}
	Let $w_1,\ldots, w_n\in \widetilde W$ be right $W_\calG$-minimal.
	There is an isomorphism
         \begin{equation}\label{Picard.group.computation.unions}
        \on{Pic}\left(\bigcup_{i=1}^n \widetilde{S}_{w_i}\right) \overset{\cong}{\longto} \bigoplus_{s} \bbZ, \;\;\; \calL\mapsto (\deg(\calL|_{S_s}))_s
        \end{equation}
        where the sum runs over all $s \in \widetilde W\backslash W_\calG$ of length $1$ such that $s\leq w_i$ for some $i=1,\ldots,n$.
        \end{proposition}
        \begin{proof}
        Without loss of generality, we may and do assume that $\bigcup_{i=1}^n \widetilde{S}_{w_i}$ is connected and contained in the neutral component $\widetilde \Fl_\calG^0$.
        Next, we proceed by induction on $n\geq 1$. 
        For $n=1$, this is \Cref{picard.group.schubert.lemma.normal}.
        For the induction step, let $X=\cup_{i=1}^{n-1}\widetilde S_{w_i}$ and $Y=\widetilde S_{w_n}$ viewed as closed subschemes of $\widetilde \Fl_\calG$. 
  	The sequence of sheaves of abelian groups on $\widetilde \Fl_\calG$
\begin{equation}\label{short.exact.units.sequence.eq}
1 \longto \iota_{X\cup Y, *}\calO_{X\cup Y}^\times \longto \iota_{X, *}\calO_{X}^\times \times \iota_{Y, *}\calO_Y^\times\overset{(a,b)\mapsto ab^{-1}}{\longto} \iota_{X\cap Y, *}\calO_{X\cap Y}^\times \longto 1 
\end{equation} is exact as is easily checked on stalks, where $\iota_{(\str)}$ denotes the respective closed immersion into $\widetilde \Fl_\calG$.
	Since $X\cap Y$ is reduced (because seminormal) by \Cref{intersections.seminormalized.Schubert.varieties}, we see $\on{H}^0(X\cap Y, \calO^\times_{X\cap Y})=k^\times$ by connectedness and projectivity of $X\cap Y$.
	Hence, the long exact (Zariski) cohomology sequence associated with \eqref{short.exact.units.sequence.eq} identifies $\Pic(X\cup Y)=\on{H}^1(X\cup Y, \calO^\times_{X\cup Y})$ with $\mathrm{Pic}(X)\times_{\mathrm{Pic}(X\cap Y)}\mathrm{Pic}(Y)$.
	One easily deduces \eqref{Picard.group.computation.unions} which finishes the induction step. 
	\end{proof}
	
	\subsubsection{Vanishing of higher coherent cohomology of seminormalized Schubert varieties} 
	Another consequence of \Cref{theorem_frobenius_split_allp} is the following result, to be used in \Cref{section_lm} below:
	
	\begin{lemma}\label{cohomology.vanishing.lemma}
	Let $w_1,\ldots, w_n\in W_{\on{af}}$ be right $W_\calG$-minimal, and consider $X=\cup_{i=1}^n\widetilde S_{w_i}$. 
	Then $\on{H}^j(X,\calO_X)=0$ for all $j\geq 1$.
        \end{lemma}
	\begin{proof}
	By \Cref{intersections.seminormalized.Schubert.varieties} finite unions of seminormalized Schubert varieties are stable under intersections.
	Hence, a Mayer--Vietoris argument similar to that in \Cref{picard.group.schubert.lemma.normal.union} reduces the claim to the case $n=1$.
	Consider the Demazure resolution $f\colon D_{\dot w}\to \widetilde S_w$ from \eqref{equation.reduced.decomposition.Demazure}. 
	Now, $\widetilde{S}_w$ has rational singularities by \Cref{theorem_frobenius_split_allp}, so $\on{H}^j(\widetilde S_w,\calO_{\widetilde S_w})=\on{H}^j(D_{\dot w},\calO_{D_{\dot w}})$ using $Rf_*\calO_{D_{\dot w}}=\calO_{\widetilde S_w}$.
	Since $D_{\dot w}$ is an iterated $\bbP^1_k$-bundle, the vanishing of higher cohomology follows by a straightforward induction argument.
	\end{proof}

	\subsection{Normality of Schubert varieties}
	In this subsection, we extend the normality theorem for Schubert varieties to some wildly ramified groups. 	
	Previously, this was proved by Faltings for split groups, see \cite[Theorem 8]{Fal03}, and by Pappas--Rapoport for Weil-restricted tame groups, see \cite[Theorem 8.4]{PR08}. 
	These results were inspired by similar ones in Kac--Moody theory found in \cite{Mat89}, but we stress that wildly ramified groups are in principle unrelated to that theory, compare with \cite[Annexe A]{Lou23}. 
	The prime-to-$p$ hypothesis on the order of $\pi_1(G_\der)$ is essential, due to \cite[Theorem 2.5]{HLR24}. 
	
	\begin{theorem}\label{thm_normality_classical_sch_vars}
		Under \Cref{hyp_odd_unitary}, all Schubert varieties $S_w$ are normal if and only if $p$ does not divide the order of $\pi_1(G_\der)$.
	\end{theorem}
	
	 We need the following auxiliary lemma:

\begin{lemma}\label{lem_reduced_flag_var}
	If $G$ is simply connected and satisfies \Cref{hyp_odd_unitary}, then $\Fl_{\calG}$ is reduced.
\end{lemma}

\begin{proof}
	This is proven in \cite[Proposition 9.9]{PR08} for tamely ramified groups and extends to wildly ramified groups under \Cref{hyp_odd_unitary}. 
	We recall the proof for convenience, following closely \cite[Proposition 9.9]{PR08}. 
	
	By \cite[Lemma 8.6]{HLR24}, it is enough to show that every $R$-valued point $x$ of $\Fl_{\calG}$, with $R$ being Artinian and strictly Henselian, factors through the reduced locus. 
	By the Bruhat decomposition and formal smoothness of $L^+\calG$, we can translate $x$ such that it is supported at the origin $e \in \Fl_{\calG}(k)$. 
	After extending scalars, we may assume that the residue field of $R$ equals $k$.
	Moreover, we can use formal smoothness of $L^+\calG$ and the fact that $R$ is strictly Henselian to lift $x$ to an $R$-valued point $\tilde{x}$ of $LG$ supported at the identity.
	This corresponds to an $R\rpot{t}$-valued point of $G$ supported at the identity, so it factors through the big cell $C=U^- \times T \times U^+$. We claim that $\tilde{x}$ is in the subgroup generated by $LU^\pm(R)$. Since the ind-schemes $LU^\pm $ are reduced, they map to $(\Fl_{\calG})_{\mathrm{red}}$. Hence, we may and do assume that $\tilde{x} \in LT$. But $T$ factors as a product of induced tori indexed by its relative coroots, and thus we can further reduce to the case when $G$ has rank $1$. Supressing the wildly ramified restrictions of scalars, then either $G=\mathrm{SL}_2$ or $\mathrm{SU}_3$ and $p \neq 2$ and the needed generation property is explicitly calculated in the proof of \cite[Proposition 9.3]{PR08}.
		So $x$ lies in the reduced locus, and the lemma follows.
\end{proof}

	\begin{proof}[Proof of \Cref{thm_normality_classical_sch_vars}]
		The seminormalization $\widetilde S_w\to S_w$ is proper and surjective, hence an isomorphism if and only if it is a monomorphism (as $S_w$ is reduced). 
		So all Schubert varieties $S_w$ are seminormal (hence normal by \Cref{seminormal.Schubert.varieties.lemma}) if and only if the morphism of ind-schemes
		\begin{equation}\label{eqn_normalization_flag_var}
		\widetilde{\Fl}_{\calG}=\on{colim} \widetilde S_w \rightarrow \on{colim} S_w =(\Fl_{\calG})_{\on{red}}\subset \Fl_\calG
		\end{equation}
		is a monomorphism, or equivalently, its restriction to the neutral components is so.  
		Using this we prove the theorem as follows.
			
			For the if clause, by \cite[Section 6.a]{PR08} we may and do assume that $G$ is simply connected, absolutely almost simple and $\calG$ is an Iwahori model. 
		In this case, we claim that \eqref{eqn_normalization_flag_var} is an isomorphism.
		Now observe that by \Cref{prop_lifts_witt}, we can find a smooth affine $W(k)\pot{t}$-group $\underline{\calG}$ with connected fibers lifting $\calG$, such that it becomes parahoric as well over $K_0\pot{t}$ with $K_0=W(k)[p^{-1}]$. Hence, \eqref{eqn_normalization_flag_var} lifts to a morphism of $W(k)$-ind-schemes
		\begin{equation}
		\widetilde{\Fl}_{\underline{\calG}}\to \Fl_{\underline{\calG}},
		\end{equation}
		where the left
		side is the ind-normalization of the right
		side. Indeed, that this commutes with base change to
		$k$ is a consequence of \Cref{theorem_frobenius_split_allp} thanks to the
		vanishing of higher coherent cohomology of the
		Demazure resolution, by an application of cohomology and base change, compare with \cite[page 52]{Fal03} and \cite[Proposition 3.13]{Gor03}. Over $K_0$, we get an isomorphism
		by Kac--Moody theory, see \cite[Section 9.f]{PR08}. Integrally, we show that the map is formally smooth around the origin, by virtue of an analogue of \cite[Lemma 10]{Fal03} or \cite[Proposition 9.3]{PR08}. This implies the claim by \cite[page 53]{Fal03} or \cite[Section 9.g]{PR08}.		
		
		The only if part follows from the argument in \cite[Section 2]{HLR24}, because if $p$ divides the order of $\pi_1(G_\der$) then the kernel of $G_{\on{sc}}\to G$ is not \'etale. 
		Hence, the induced morphism $\Fl_{\calG_{\on{sc}}}\to \Fl_{\calG}^0\subset \Fl_\calG$ is not a monomorphism, where $\calG_{\on{sc}}$ denotes the parahoric $O$-model of $G_{\on{sc}}$ induced by $\calG$.
		By \Cref{lem_reduced_flag_var}, $\Fl_{\calG_{\on{sc}}}$ is reduced, so \eqref{eqn_normalization_flag_var} factors on neutral components as $\widetilde{\Fl}_{\calG}^0 \overset{\sim}{\to}\Fl_{\calG_{\on{sc}}}\to (\Fl_{\calG})_{\on{red}}^0$.
		Now, if \eqref{eqn_normalization_flag_var} were a monomorphism, then $\Fl_{\calG_{\on{sc}}}\to \Fl_{\calG}^0$ would be a monomorphism, which is a contradiction.
	\end{proof}	
	
	\subsection{Central extensions of line bundles}
	
	In the theory of loop groups and their flag varieties, one is usually faced with the obstacle that not every line bundle on $\Fl_{\calG}$ is $LG$-equivariant. However, this can partially remedied by considering a certain universal central extension of $LG$ that acts on every line bundle of $\Fl_{\calG}$. This is a recurrent theme in Kac--Moody theory, see \cite[page 54]{Fal03}, \cite[Remark 10.2]{PR08} and \cite[Corollary 4.3.11]{Lou23}, and also admits an incarnation for the Witt vector Grassmannian by \cite[Proposition 10.3]{BS17}. In order to properly explain it, we need to use the geometric results of the previous subsections.
		
		Given a line bundle $\calL$ on $\Fl_{\calG}$, we form the group functor on the category of $k$-algebras $R$ defined by
	\begin{equation}
	LG\{\calL\}(R) \,=\, \{(g,\al)\,|\,g\in LG(R), \;\al\colon \calL\cong g^*\calL\}.
	\end{equation}
	We can now prove the following lemma:
	
	\begin{lemma}\label{central.extension.lemma}
		Suppose $G$ is an almost simple, simply connected $K$-group satisfying \Cref{hyp_odd_unitary}. Then, the pre-sheaf $LG\{\calL\}$ defines a central extension of $LG$ by $\bbG_{m,k}$ in the category of ind-affine $k$-group ind-schemes.
		The association $\calL\mapsto LG\{\calL\}$ induces a group homomorphism
		\begin{equation}\label{central.extension.map.equation}
		\Pic(\Fl_{\calG}) \to \Ext_{\on{cent}}(LG, \bbG_{m,k}). 
		\end{equation}
		with the same kernel as \eqref{equation.central.charge} restricted to $\Pic(\Fl_{\calG})$.
	\end{lemma}

\begin{proof}
  Note that $LG \{\calL\}(R)$ carries a natural group structure via
  $(g_1,\alpha_1)\cdot (g_2,\alpha_2)=(g_1g_2, g^*_2\alpha_1 \circ \alpha_2 )$, thus having
  $\bbG_{m,k}(R)=\{(1,c)\,|\, c \in R^{\times}\}$ as a central
  subgroup. We claim moreover that
  $\bbG_{m,k}(R) \subset LG \{\calL\}(R)$ is the kernel of the natural
  projection to $LG(R)$. In other words, we claim that the
  automorphism group $\mathrm{Aut}(\calL_R)$ as a line bundle on
  $\Fl_{\calG,R}$ equals $R^{\times}$. After tensoring with
  $\calL^{-1}$, we may and do assume that $\calL=\calO$. Thus, it
  suffices to show that $H^0(\Fl_{\calG,R},\calO)=R$ which is implied
  by \Cref{lem_reduced_flag_var}.

        Next, we study the action of $LG(R)$ on the Picard
        groups. Note that $\mathrm{Pic}(\Fl_{\calG,R})$ is the direct
        sum of $\mathrm{Pic}(R)$ and $\mathrm{Pic}(\Fl_{\calG})$,
        since the Picard functor of the flag variety is constant
        \'etale due to \cite[Corollary 5.13]{Kle05} using \Cref{cohomology.vanishing.lemma}. The action
        of $LG(R)$ on $\mathrm{Pic}(R)$ is trivial, and we claim that
        the same holds for the quotient
        $\mathrm{Pic}(\Fl_{\calG,R})/\mathrm{Pic}(R)$. By
        \Cref{thm_normality_classical_sch_vars} and
        \eqref{Fl_Picard.group.computation.normal}, that quotient is
        torsion-free and we may check triviality of the $LG(R)$-action
        on generators of the associated $\bbQ$-vector space. A set of
        generators is given by $LG$-equivariant line bundles, see
        \Cref{lemma.equivariant.lb}, and the adjoint line bundle. For
        an $LG$-equivariant line bundle, the claim is trivial and we
        even see directly that $LG\{\calL\} \to LG$ splits and thus is
        the trivial extension. For the adjoint line bundle, one sees
        that the difference \begin{equation}
          \calL^{-1}_\mathrm{det}\cdot
          g^*\calL_{\mathrm{det}}=\mathrm{det}(t^{-a}R\pot{t}^n/gR\pot{t}^n)
          \in \mathrm{Pic}(\Fl_{\mathrm{SL}_n,R})
      \end{equation}
      for $a\gg 0$ is in
      the image of $\mathrm{Pic}(R)$, compare \cite[page 43]{Fal03}, so the same remains true after pulling back to
      $\Fl_{\calG,R}$.

		We can use the previous paragraph to show that any
                $R$-valued point of $LG$ lifts along the map
                $LG\{\calL\} \to LG$ after we replace $\Spec\, R$ by a finite union of affine opens.
                Indeed, we saw above that $\calL$ and $g^*\calL$ differ by an element of $\mathrm{Pic}(R)$ which can be trivialized over an affine open $\Spec\,  S \subset \Spec\,  R$. Replacing $R$ by $S$, we may assume the existence of an isomorphism $\alpha\colon \calL \cong g^*\calL$, thereby producing a lift in $LG\{\calL\}(S)$. Letting $\Spec\,  R$ run over sufficiently small affine opens of a presentation of $LG$, the existence of lifts shows that $LG\{\calL\}$ is representable by an ind-affine $k$-group ind-scheme and that it is an extension of $LG$ by $\bbG_{m,k}$. Finally, it is clear that the kernel of \eqref{central.extension.map.equation} consists of those $\calL$ that admit an $LG$-equivariant structure, hence coincides with the kernel of \eqref{equation.central.charge} after restricting the latter to $\mathrm{Pic}(\Fl_{\calG})$ thanks to \Cref{lemma.equivariant.lb}.
                
                
\end{proof}

The lemma implies that the image of \eqref{central.extension.map.equation} is a free $\bbZ$-module of rank $1$, see \Cref{lemma.kac.moody.coeff.inj}. 
Identify the image with $\bbZ$ via the unique isomorphism sending ample line bundles to positive integers. 

\begin{corollary}
For any $\calL\in \mathrm{Pic}(\Fl_\calG)$ with $c_\calL=1$, the resulting central extension $\widehat{LG}:=LG\{\calL \}$ has the property that every line bundle on $\Fl_\calG$ admits a $\widehat{LG}$-equivariant structure which is unique up to multiplication by $k^\times$.
\end{corollary}
\begin{proof}
Using \Cref{lemma.equivariant.lb}, the central charge induces a short exact sequence $0\to \Pic([LG\backslash \Fl_\calG])\to \Pic(\Fl_\calG)\overset{c}{\to} \bbZ\to 0$. 
The choice of $\calL$ provides a splitting. 
So the corollary follows from the equality $\on{Aut}(\calM)=k^\times$ for any line bundle $\calM$ on $\Fl_\calG$, see the proof of \Cref{central.extension.lemma}.
\end{proof}

	\section{Local models} \label{section_lm}

	In this final section, let $O$ be a complete discretely valued ring with fraction field $K$ and perfect residue field $k$ of characteristic $p>0$.
	Let $G$ be a reductive $K$-group, $\mu$ a (not necessarily minuscule) geometric conjugacy class of cocharacters in $G$ and $\calG$ a parahoric $O$-model of $G$.
	The reflex field $E$ of $\mu$ is a finite separable field extension of $K$ with ring of integers $O_E$ and residue field $k_E$.
	
	Let $\breve O$ be the completed strict Henselisation of $O$ with fraction field $\breve K$ and algebraically closed residue field $\bar k$. 
	Let $T$ be the centralizer of some maximal $\breve K$-split torus $S$ which is defined over $K$ and contains a maximal $K$-split torus with apartment containing the facet associated with $\calG$, see \cite[Corollaire 5.1.12]{BT84}.
	The connected N\'eron model $\calT$ of $T$ is a closed subgroup scheme of $\calG$.

	\subsection{Equicharacteristic local models} \label{equichar_LM_sec}
	Assume $K=k\rpot{t}$ is a Laurent series field with ring of integers $O=k\pot{t}$.
	Let us recall the definition of local models in equicharacteristic, which only depend on the pair $(\calG,\mu)$ and not on additional auxiliary choices. 
	Recall that we have defined the Beilinson--Drinfeld Grassmannian $\Gr_{\calG} \to \Spec \,O$ before \Cref{line.bundle.extension}.
	Its generic fiber is equivariantly isomorphic to the affine Grassmannian $\Gr_G\to \Spec\, K$ whereas its special fiber is equal to the affine flag variety $\Fl_\calG\to \Spec\, k$. 
	Let $S_{G,\mu}\subset \Gr_G\times_{\Spec\, K}\Spec\, E$ be the Schubert variety attached to $\mu$.
	
	\begin{definition}\label{loc_mod_def}
		Let $M_{\calG, \mu}$ denote the flat closure of $S_{G,\mu}$ inside the Beilinson--Drinfeld Grassmannian $\Gr_{\calG,O_E}:=\Gr_{\calG}\times_{\Spec\, O}\Spec\, O_E$. 
		We denote by $\widetilde{M}_{\calG, \mu}$ its seminormalization \cite[0EUK]{StaProj}.
              \end{definition}

	\begin{remark}\label{functoriality_remark}
	The formation of orbit closures and their seminormalizations are functorial in the following sense.
	A morphisms of pairs $(\calG,\mu)\to (\widetilde \calG,\widetilde \mu)$ is a map of $O$-group schemes $\calG\to \widetilde \calG$ which maps $\mu$ into $\widetilde \mu$ under the induced map of reductive $K$-groups $G\to \widetilde G$ in the generic fiber.
	Any such map of pairs induces a map $M_{\calG,\mu}\to M_{\widetilde \calG, \widetilde \mu}$ commuting over $\Spec\, O_E\to \Spec\, O_{\widetilde E}$ where $\widetilde E$ denotes the reflex field of $\widetilde \mu$.
	By functoriality of seminormalizations \cite[Tag
        0EUS]{StaProj}, we get a map $\widetilde M_{\calG,\mu}\to \widetilde M_{\widetilde \calG, \widetilde \mu}$ commuting over the map of orbit closures.
	\end{remark}

	In order to describe the special fiber of the schemes from \Cref{loc_mod_def}, we
	recall the admissible locus \cite[Section 4.3]{PRS13}.
	The Kottwitz homomorphism induces an isomorphism $X_*(T)_I\cong T(\breve K)/\calT(\breve O), \bar\la\mapsto \bar\la(t)$ where the source denotes the coinvariants of the cocharacter lattice $X_*(T)$ under the inertia subgroup $I$ of the absolute Galois group of $K$. 
	Note that the isomorphism does not depend on the choice of uniformizer $t$.

	\begin{definition} \label{def_admissible} The admissible locus
          $A_{\calG, \mu}$ is the reduced $k_E$-subscheme of
          $\Fl_{\calG, k_E}$ given by the $k_E$-descent of the union
          of $\bar k$-Schubert varieties $S_{\bar\la(t)}$, where
          $\la\in X_*(T)$ runs through the (finitely many)
          representatives of $\mu$ and where $\bar\la\in X_*(T)_I$
          denotes its image in the coinvariants under $I$.  We denote
          by $\widetilde{A}_{\calG, \mu}$ its seminormalization.
	\end{definition}
	
	Note that $A_{\calG,\mu}$ does not depend on the choice of the
        maximal torus $T$ as above.  Further, $A_{\calG,\mu}$ is
        geometrically connected and, by \cite[Theorem 4.2]{Hai18}, its
        irreducible $\bar k$-components are the Schubert varieties
        $S_{\bar\la(t)}$ where $\la$ runs through the
        $\breve{K}$-rational representatives of $\mu$ in $X_*(T)$.

	Let us now discuss finer geometric properties. It was shown in
        \cite[Theorem 6.12]{HR21} that the reduced special fiber of
        $M_{\calG, \mu}$ coincides with $A_{\calG, \mu}$, but we shall
        only need to use the inclusion of $A_{\calG, \mu}$ in the
        reduced special fiber, already proved in \cite[Lemma
        3.12]{Ric16}. Note that
        $(\widetilde{A}_{\calG, \mu})_{\bar{k}} = \cup_{\lambda}
        \widetilde{S}_{\bar{\lambda}(t)}$ by
        \Cref{intersections.seminormalized.Schubert.varieties}, where
        $\lambda$ ranges over the $\breve{K}$-rational representatives
        of $\mu$ in $X_*(T)$.  Since the $F$-split property for proper
        schemes can be descended from $\bar{k}$ to $k_E$,
        $\widetilde{A}_{\calG, \mu}$ is $F$-split.  It identifies
        moreover with the admissible locus
        $A_{\widetilde{\calG}, \widetilde{\mu}}$ associated with any
        $z$-extension $\widetilde{G}$ of $G$ with simply connected
        derived group, and any lift $\widetilde{\mu}$ of $\mu$, by
        \Cref{thm_normality_classical_sch_vars}, at least when
        \Cref{hyp_odd_unitary} holds.

	Now, we may state our main result on the singularities of local models.
	
	\begin{theorem}\label{theorem_coherence_allp}
		Under \Cref{hyp_odd_unitary}, the local
		model $\widetilde{M}_{\calG, \mu}$ is Cohen--Macaulay, has $F$-rational singularities (and thus is pseudo-rational), and has reduced special fiber
		equal to the seminormalized admissible locus
		$\widetilde{A}_{\calG,\mu}$.
	\end{theorem}

	\begin{proof}
          The key step of the proof is showing that the special fiber
          is reduced and equal to $\widetilde{A}_{\calG,\mu}$ from
          which the other properties follow by using the $F$-splitness
          of $\widetilde{A}_{\calG,\mu}$; in fact, we shall prove that
          $\widetilde{M}_{\calG, \mu}$ has $F$-rational singularities.
          This part of the proof essentially follows from
          \cite[Section 4.2]{Zhu14}, relying on
          \Cref{thm_normality_classical_sch_vars} for wildly ramified
          groups.  Here is an outline.  By using faithfully flat
          descent of $F$-rationality \cite[Proposition A.5]{DM20} we
          may reduce to the case $O=\breve O$, so $G$ is quasi-split.

          First, we show that for any finite field extension
          $\widetilde E/E$ the base change
          $\widetilde{M}_{\calG,
            \mu}\otimes_{\calO_E}\calO_{\widetilde E}$ is normal with
          reduced special fiber equal to $\widetilde{A}_{\calG,\mu}$
          as follows.  Passage to the adjoint group induces a map of
          pairs $(\calG,\mu)\to (\calG_{\on{ad}},\mu_{\on{ad}})$ where
          $\calG_{\on{ad}}$ is the parahoric associated with
          $G_{\on{ad}}$ and $\mu_\ad$ is induced by $\mu$ under
          $G\to G_\ad$.  The corresponding map
          $\widetilde M_{{\calG},\mu}\to\widetilde
          M_{{\calG_{\on{ad}}},\mu_\ad}\otimes_{O_{E_\ad}}O_E$ is a
          universal homeomorphism inducing isomorphisms on residue
          fields by \cite[Corollary 2.3 and its proof]{HR22}, thus an
          isomorphism if the target is (semi-)normal.  Without loss of
          generality, we reduce to the case where $G$ is adjoint.  A
          similar argument shows that the formation of
          $\widetilde{M}_{\calG, \mu}$ commutes with products in $G$,
          so we first assume that $G$ is adjoint and
          simple, so $G=\Res_{L/K}(G_0)$ for a finite separable field
          extension $L/K$ (necessarily totally ramified) and an
          absolutely simple $L$-group $G_0$.

          The simply connected cover
          $G_{\on{sc}}\to G$ induces a universally closed and
          universally injective morphism
          $\iota\colon \Gr_{\calG_{\on{sc}}}\to \Gr_{\calG}$ which
          gives on generic fibers the universal homeomorphism
          $\Gr_{G_{\on{sc}}}\to \Gr_G^0$ onto the neutral component.
          We consider the translate
          $t_{\mu}^{-1}M_{\calG,\mu} \subset
          \iota(\Gr_{\calG_{\on{sc}}, O_E})$, where $t_{\mu}$ is an
          $O_E$-valued point of $L\calT$ lifting the corresponding
          section of $\Gr_\calT$, and consider the unique reduced
          closed subscheme
          $M_{\calG_{\on{sc}}, \mu}\subset \Gr_{\calG_{\on{sc}}, O_E}$
          with the topological space $\iota(M_{\calG_{\on{sc}}, \mu})$
          being the same as the translation.  Likewise, we denote by
          $A_{\calG_{\on{sc}},\mu}$ (respectively,
          $S_{G_{\on{sc}},\mu}$) the $t_\mu$-translated admissible
          locus inside $\Fl_{\calG_{\on{sc}}}$ (respectively,
          $\Gr_{G_{\on{sc}},E}$). These are also unions of translates
          of Schubert varieties for some choice of Iwahori group
          scheme.  The induced finite universal homeomorphism
          $M_{\calG_{\on{sc}}, \mu}\to M_{\calG,\mu}$ factors on
          generic fibers as
          $S_{G_{\on{sc}},\mu}\cong \widetilde S_{G,\mu}\to S_{G,\mu}$
          (hence is birational).

          We will prove that for all $n \geq 1$, we have
          \begin{equation}\label{equation.determinant.line}
            \dim_k H^0(A_{\calG_{\on{sc},\mu}},
            \calL_{\on{ad}}^{\otimes n})
            =\dim_{E} H^0(S_{G_{\on{sc}}, \mu},
            \calL_{\on{ad}}^{\otimes n}),
          \end{equation}
           where $\calL_{\on{ad}}$ denotes
          the pullback of the determinant line bundle along the
          adjoint representation, compare
          \Cref{lem.central.charge.adjoint}.  But before we do so let
          us explain how it implies that $\widetilde{M}_{\calG,\mu}$
          has special fiber equal to $\widetilde{A}_{\calG,\mu}$.  By
          \cite[Lemma 3.12]{Ric16}, we have an inclusion of
          $A_{\calG_{\on{sc}},\mu}$ in the reduced special fiber of
          $M_{\calG_{\on{sc}}, \mu}$. Since $\calL_{\on{ad}}$ is a
          relatively ample line bundle on $M_{\calG_{\on{sc}}, \mu}$,
          \eqref{equation.determinant.line} implies that the special
          fiber of $M_{\calG_{\on{sc}}, \mu}$ is reduced and equal to
          $A_{\calG_{\on{sc}},\mu}$. By Serre's criterion (see
          \cite[Proposition 9.2]{PZ13}) it follows that
          $M_{\calG_{\on{sc}}, \mu}$ is normal.  Consequently, as the
          map $M_{\calG_{\on{sc}}, \mu}\to M_{\calG,\mu}$ induces an
          isomorphism on every residue field, it identifies with the
          seminormalization, so induces an isomorphism
          $M_{\calG_{\on{sc}}, \mu}\cong \widetilde M_{\calG, \mu}$.
          Using the normality of Schubert varieties for simply
          connected groups in \Cref{thm_normality_classical_sch_vars}
          we then see that the special fiber of
          $\widetilde M_{\calG, \mu}$ is
            $\widetilde A_{\calG,\mu}$.

		It remains to prove \eqref{equation.determinant.line}.
For this, consider the $W(k)\pot{t}$-lift $\underline{\calG_{\on{sc}}}$ of $\calG_{\on{sc}}$ provided by \Cref{prop_lifts_witt} under our \Cref{hyp_odd_unitary}, which holds for $\Phi_{G_{\on{sc}}}$.
		Consider the affine flag scheme $\Fl_{\underline{\calG_{\on{sc}}}}$ over $W(k)$.
		It admits the flat, closed subscheme $A_{\underline{\calG_{\on{sc}}},\mu}$ whose generic fiber is $A_{\calG'_{\on{sc}}, \mu'}$ with $\calG'_{\on{sc}}=\underline{\calG_{\on{sc}}}\otimes K_0\pot{t}$ and $\mu'$ corresponding to $\mu$ using \eqref{equation.global.IW}, and whose special fiber contains $A_{\calG_{\on{sc}},\mu}$. 
		As explained in the last paragraph of the proof of \cite[Th\'eor\`eme 5.2.1]{Lou23}, one deduces from the combinatorics of Schubert varieties and their compatible $F$-splitness an equality  
		\begin{equation}\label{equation.characteristic.zero}
		\dim_k H^0(A_{\calG_{\on{sc}},\mu}, \calL_{\on{ad}}^{\otimes n})=\dim_{K_0} H^0(A_{\calG'_{\on{sc}},\mu'}, \calL_{\on{ad}}^{\otimes n}),
		\end{equation}
		for all $n\geq 1$.
		Note that \eqref{equation.characteristic.zero} uses again the normality of Iwahori Schubert varieties for simply connected groups (\Cref{thm_normality_classical_sch_vars}) to deduce their $F$-splitness (\Cref{theorem_frobenius_split_allp}). 
		Likewise, the analogue of \eqref{equation.characteristic.zero} also holds for $S_{G_{\on{sc}},\mu}$ versus $S_{G_{\on{sc}}',\mu'}$ with $G_{\on{sc}}'=\calG'_{\on{sc}}\otimes K_0\rpot{t}$.
		Appealing now to the coherence theorem of \cite{Zhu14}
                for the group $G'_{\on{sc}}$ in characteristic $0$
                (those are always tamely ramified) finishes the proof of \eqref{equation.determinant.line}.
			Thus, $\widetilde{M}_{\calG, \mu}$ is normal and has
                reduced special fiber which is equal to
                $\widetilde{A}_{\calG, \mu}$, and the same holds for
                the base change $\widetilde{M}_{\calG,
                  \mu}\otimes_{O_E}O_{\widetilde E}$ by an application
                of Serre's criterion as the generic fiber is
                geometrically normal.
		
                 Since, as noted above, the formation of
                  $\widetilde{M}_{\calG, \mu}$ commutes with products
                  in $G$, it follows that for general $G$ the special
                  fiber of $\widetilde{M}_{\calG, \mu}$ is reduced and
                  is equal to $\widetilde{A}_{\calG, \mu}$.  We now
                prove the other parts of the theorem by using results
                from the theory of $F$-singularities, see
                \Cref{Fsing_review}. Since
                $\widetilde{S}_{G,\mu, \bar K}\cong S_{G,\mu, \bar
                  K}^{\on{sc}}$ is an Iwahori Schubert variety for the
                simply connected, split reductive group
                $G_{\on{sc},\bar K}$, it is Cohen--Macaulay and even
                $F$-rational by \cite[Theorem
                1.4]{Cas22}. (Alternatively, these properties of
                $\widetilde{S}_{G,\mu,\bar K}$ also follow directly
                from \Cref{theorem_frobenius_split_allp}.)  Hence, so
                is $\widetilde{S}_{G,\mu}$ by faithfully flat descent
                of \cite[Proposition A.5]{DM20}.  We already know that
                $\widetilde{A}_{\calG, \mu}$ is $F$-split by
                \Cref{theorem_frobenius_split_allp}, so it is
                $F$-injective in particular.  We also note that all
                rings and schemes involved in our argument are
                $F$-finite since $k$ is algebraically closed.  Then
                \Cref{lemma.Schwede.Singh} implies that
                $\widetilde{M}_{\calG, \mu}$ is $F$-rational, so
                pseudo-rational by \Cref{Smith97} and in particular
                Cohen--Macaulay.
	\end{proof}
	
	\begin{remark}\label{rem:good primes equality}
		There is an equality $\widetilde{M}_{\calG, \mu}= M_{\calG, \mu}$ if and only if $\widetilde{A}_{\calG, \mu}= A_{\calG, \mu}$ and $\widetilde{S}_{G,\mu}=S_{G,\mu}$. This is ensured, for instance, when $p \nmid \lvert \pi_1(G_\der)\rvert$. 
		If $p \mid  \lvert \pi_1(G_\der)\rvert$, then the equality still holds when $\bar{\mu} \in X_*(T)_I$  is minuscule with respect to the \'{e}chelonnage roots and the closure of ${\bf f}$ contains a special vertex; see the proof of \cite[Proposition 9.1]{HLR24}. 
		Otherwise the equality is false for infinitely many values of $\mu$, see \cite[Corollary 9.2]{HLR24}.
	\end{remark}
	
	\begin{remark}
		Cass has proved somewhat stronger properties of the
		singularities of $\widetilde{M}_{\calG, \mu}$ when the group
		$G$ is a constant split reductive group and $p>2$, see \cite[Theorem
		1.6]{Cas21}.
	\end{remark}

\begin{remark}\label{remark.GL}
	There is an alternative proof for the reducedness of the special fiber of $\widetilde{M}_{\calG,\mu}$ via perfectoid geometry, see \cite[Lemma 1.2, Theorem 1.3]{GL24}, without the need for \Cref{hyp_odd_unitary}. We stress that it does not directly imply that the special fiber is seminormal and $F$-split as in \Cref{theorem_coherence_allp}, upon which the last sentence of \cite[Corollary 1.4]{GL24} actually relies. On the other hand, combining the results of \cite{GL24} with \Cref{thm_normality_classical_sch_vars} immediately yields an identification between $\widetilde{A}_{\calG,\mu}$ and the special fiber of $\widetilde{M}_{\calG,\mu}$, compare with the proof of \cite[Theorem 2.1]{GL24} or the discussion surrounding \cite[Conjecture 7.25]{AGLR22}.
\end{remark}

We can also deduce the following facts on the Picard group of the local models.

\begin{corollary}\label{cor_pic_local_model_equal}
Under \Cref{hyp_odd_unitary}, the following properties hold:
\begin{enumerate}
	\item 
	\label{cor_pic_local_model_equal.1}
	The restriction map
	$\mathrm{Pic}(\widetilde{M}_{\calG,\mu}) \to
	\mathrm{Pic}(\widetilde{A}_{\calG,\mu})$ is an isomorphism.
	\item 
	\label{cor_pic_local_model_equal.2}
	Let $G_i$ for $i=1,\dots, m$ be an enumeration of the simple factors of $G_\mathrm{ad}$ such that the image $\bar \mu_i$ of $\mu$ in the group $X_*(T_i)_I$ attached to $G_i$ is non-zero. 
	Then the restriction map 
	\begin{equation}
          \prod_{i=1}^{m}\mathrm{Pic}(\widetilde{\Fl}^{\tau_{i}}_{\calG_{i}})
          \to \mathrm{Pic}(\widetilde{A}_{\calG,\mu}) \end{equation}
        is an isomorphism, where $\calG_i$ is the associated parahoric
        $O$-model of $\calG_i$ and the superscript $\tau_i$ indicates the connected component
        attached to $\mu_i$.
	\item 
	\label{cor_pic_local_model_equal.3}
	There is a commutative diagram:
	\begin{equation}
	\begin{tikzcd}[column sep=4.5cm, row sep=1cm,ampersand replacement=\&]
	 \mathrm{Pic}(\widetilde{M}_{\calG,\mu}) \ar[r, ""] \ar[d, "\sim"]
	\& \mathrm{Pic}(\widetilde{S}_{G,\mu})
	\ar[d, "\sim"] \\
	\mathrm{Pic}(\widetilde{A}_{\calG,\mu})  \& \prod_{i=1}^{m}\mathrm{Pic}(\widetilde{S}_{G_i,\mu_i})\ar[d, "\prod_{i=1}^m \mathrm{deg}_i"]
	\\
\prod_{i=1}^{m}\mathrm{Pic}(\widetilde{\Fl}^{\tau_{i}}_{\calG_{i}}) \ar[u, "\sim"] \ar[r, "\prod_{i=1}^m c_i"] \& \bbZ^m ,
	\end{tikzcd}
	\end{equation}
	where the maps of Picard groups are induced by functoriality, $\mathrm{deg}_i$ denotes the degree homomorphism, and the $c_i$ are the central charge homomorphisms for $\Fl_{\calG_{i,\on{sc}}}$ translated to the respective connected components.	
\end{enumerate} 	
\end{corollary}

\begin{proof}

By \Cref{theorem_coherence_allp}, the special fiber of $\widetilde{M}_{\calG,\mu}$ is equal to $\widetilde{A}_{\calG,\mu}= \cup_{\lambda} \widetilde{S}_{\bar{\lambda}(t)}$, see \Cref{def_admissible}. 
For \eqref{cor_pic_local_model_equal.1}, it is enough to prove that every line bundle on $\widetilde{A}_{\calG,\mu}$ lifts uniquely to $\widetilde{M}_{\calG,\mu}$, or equivalently to the formal scheme $\widetilde{M}_{\calG,\mu}\times_{\Spec(O_E)}\on{Spf}(O_E)$ by Grothendieck's formal GAGA. 
Since $\on{H}^j(\widetilde{A}_{\calG,\mu},\calO_{\widetilde{A}_{\calG,\mu}})=0$ for $j=1,2$ by \Cref{cohomology.vanishing.lemma}, obstruction theory (compare \cite[Proposition 5.19]{Kle05}) shows the existence and uniqueness of such lifts. 

For \eqref{cor_pic_local_model_equal.2}, we may and do assume that $k$ is algebraically closed by \'etale descent. We use \Cref{picard.group.schubert.lemma.normal.union} which calculates $\on{Pic}(\widetilde{A}_{\calG,\mu})$ as $\oplus_s \bbZ$ where the sum runs over all $s \in \widetilde{W} \setminus W_\calG$ with $l(s)=1$ and $s \leq \bar\lambda(t)$ for some rational representative $\lambda$ of $\mu$ in $X_*(T)$. 
        In order to finish the proof of the second part, we may and do assume that $G$ is simple and $\mu$ is non-zero. 
        We have to show that the map
        $\mathrm{Pic}(\widetilde \Fl^{\tau_\mu}_{\calG})\to \mathrm{Pic}(\widetilde{A}_{\calG,\mu})$ 
        is an isomorphism where $\tau_\mu$ denotes the unique length $0$ element in the admissible set $\mathrm{Adm}(\mu)\subset \widetilde{W}$.
        It is enough to show that every simple reflection $s\in W_{\on{af}}$ appears in $\tau_\mu^{-1}\mathrm{Adm}(\mu)$, see \Cref{picard.group.flag.variety.lemma.normal}. 
        Assume the contrary.
        Then the subgroup generated by the simple reflections which do appear is a finite Coxeter group, say, $W'$ containing $\tau_\mu^{-1}\mathrm{Adm}(\mu)$.
        Therefore, $W'$ (hence $\tau_\mu^{-1}\mathrm{Adm}(\mu)$) contains at most one representative for each coset in the finite Weyl group $W_0=W_{\on{af}}/X_*(T_{\on{sc}})_I$: if there were two representatives, their difference would be a non-trivial translation, so $W'$ would not be finite. However, this contradicts the fact that $\mathrm{Adm}(\mu)$ contains always at least two different translations $t_{\bar \mu}$ and $t_{w_0(\bar \mu)}$ because $\bar \mu\neq 0 $.

        Part \eqref{cor_pic_local_model_equal.3} is verified as follows. Since the groups involved are all torsion-free, we only need to check commutativity after tensoring with $\bbQ$. But then \Cref{line.bundle.extension} applied to each of the simple factors provides rationalized line bundles on $\widetilde{M}_{\calG,\mu}$ whose generic fiber is given by $\calO(c_{\calL_k})$, exactly as claimed.        
\end{proof}

	Recall that Pappas--Rapoport's coherence conjecture in \cite{PR08}, as corrected by Zhu in \cite{Zhu14}, gives an equality of dimensions of certain cohomology groups, which we can now formulate and prove in greater generality.

	\begin{corollary}\label{corollary.coherence.conjecture}
		Let $\calL$ be an ample line bundle on $\widetilde{A}_{\calG,\mu}$.
		Under \Cref{hyp_odd_unitary}, there is an equality
		\begin{equation}\label{equation.coherence.conjecture}
		\dim_k \on{H}^0(\widetilde{A}_{\calG,\mu}, \calL)=\dim_{\bar K} \on{H}^0(\widetilde{S}_{G, \mu,\bar K}, \calO(c_\calL)),
		\end{equation}
		 where $\calO(c_\calL):=\boxtimes_i \calO(c_i(\calL))$ and the $c_i$ are the central charge homomorphisms of the simple factors of $G_{\mathrm{ad}}$, compare with \Cref{cor_pic_local_model_equal}.
	\end{corollary}

\begin{proof}
	Note that given a flat proper scheme $X$ over a discrete valuation ring with $F$-split special fiber, and an ample line bundle $\calL$ on $X$, the dimension of the global sections of $\calL$ on $X_s$ and $X_\eta$ agree by the vanishing of higher cohomology (and constancy of the Euler characteristic). Therefore, the statement follows directly from \Cref{theorem_frobenius_split_allp}, \Cref{theorem_coherence_allp}, and \Cref{cor_pic_local_model_equal}. Indeed, by \Cref{cor_pic_local_model_equal} \eqref{cor_pic_local_model_equal.1}, $\calL$ lifts uniquely to an ample line bundle over $\widetilde{M}_{\calG,\mu}$ with geometric generic fiber equal to $\calO(c_\calL)$ by \Cref{cor_pic_local_model_equal} \eqref{cor_pic_local_model_equal.3} (note that the integers $c_i(\calL)$ are well defined by \Cref{cor_pic_local_model_equal} \eqref{cor_pic_local_model_equal.2}).
	\end{proof}

	\subsection{Mixed characteristic}\label{sec:mixed characteristic local models}
	In this subsection, we assume $K/\bbQ_p$ is of characteristic $0$, and fix a uniformizer $\pi\in K$. 
	Further, $G$ is assumed to be adjoint, quasi-split and to satisfy \Cref{hyp_odd_unitary}. 
	Then $G$ is a product of $K$-simple groups compatibly with the tori $S\subset T$, and we fix the data in \eqref{equation.splitting.map} for each factor.
	The resulting $O\pot{t}$-group lift $\underline{\calG}$ of its parahoric model $\calG$ is defined as the product of the lifts from \Cref{prop_lifts_bk} of each simple factor. 
	We denote $\calG':=\underline\calG\otimes k\pot{t}$.
	
	Let us recall the basic properties of the Beilinson--Drinfeld Grassmannian $\Gr_{\underline{\calG}}\to \Spec\, O$, where the power series variable is $z=t-\pi$, compare with \cite[Section 6]{PZ13}.
	
	\begin{proposition}\label{prop_ident_BD_Grass}
		The $O$-functor $\Gr_{\underline{\calG}}$ is representable by an ind-projective ind-scheme. 
		Its generic fiber is isomorphic to $\Gr_{G}$, whereas the special fiber is identified with $\Fl_{\calG'}$.
	\end{proposition}
	
	\begin{proof}
		Representability by an ind-quasi-projective ind-scheme follows from \cite[Proposition 11.7]{PZ13}, thanks to \Cref{prop_lifts_witt}. 
		Its special fiber is the affine flag variety associated to the $k\pot{t}$-group scheme $\calG'$, that is, $\Gr_{\underline{\calG},k}=\Fl_{\calG'}$.
		As for the generic fiber, we have to find and choose an identification between $\underline{\calG}\otimes K\pot{z}$ and $G\otimes K\pot{z} $.
		But the former group scheme is reductive, so such an isomorphism exists by \cite[Lemma 0.2]{Ric19b}, which says that every reductive group scheme over $K\pot{z}$ is constant.
		
		Finally, we show projectivity by the same argument of
                \cite[Proposition 6.5]{PZ13}: it is enough to verify
                the valuative criterion for
                $\Gr_{\underline{\calT}}$. Since $\underline{\calT}$ is a product of
                restrictions of scalars of the multiplicative group along maps of the smooth $O$-curves in \eqref{equation.mixed.lifts}, this is a consequence of \cite[Corollary 3.6, Lemma 3.8]{HR20b}.
	\end{proof}
	
	Just as in \cite[Section 7]{PZ13}, we introduce local models in mixed characteristic. 
	\begin{definition}
		Let $M_{\underline{\calG}, \mu}$ denote the flat closure of $S_{G,\mu}$ inside $\Gr_{\underline{\calG}, O_E}$. 
		We denote by $\widetilde{M}_{\underline{\calG}, \mu}$ its seminormalization.
	\end{definition}
	
	The reader is referred to \Cref{adjoint.groups.general.definition} for the extension to not necessarily adjoint groups and to \Cref{z.extensions.LM} for the relation to the (modified) local models from \cite[Section 2.6]{HPR20}.
	In the following paragraphs, we single out some important properties of the local models. 
	
            \begin{lemma} \label{l.inclusion}%
              The reduced special fiber of
              $M_{\underline{\calG}, \mu}$ contains the
              $\mu'$-admissible locus $A_{\calG', \mu'}$ in
              equicharacteristic, where $\mu'$ is the corresponding
              dominant absolute coweight of $G'$.
            \end{lemma}
            \begin{proof}
              The proof is the same as the proof of \cite[Lemma
              3.12]{Ric16}. This depends on \cite[Lemma 2.21]{Ric16}
              which is formulated in an equicharacteristic setting,
              but the proof extends to the mixed characteristic
              setting using that $\underline \calT$ is induced.
            \end{proof}

	\begin{remark}
		Since our group lifts seldom coincide with the corresponding
		constructions of \cite{Lev16},
		we do not know how to compare our $M_{\underline{\calG},\mu}$ and $\widetilde{M}_{\underline{\calG},\mu}$ with the local models from \cite{Lev16}, when $\mu$ is non-minuscule. 
		However, our arguments and results below still hold for both objects.
		For minuscule $\mu$, both constructions do coincide by \cite[Section 7]{AGLR22}.
	\end{remark}
	
	Now, we may state our main result on the singularities of local models.
	
	\begin{theorem}\label{theorem_coherence_mixed_char}
		Under \Cref{hyp_odd_unitary}, the local model $\widetilde{M}_{\underline{\calG}, \mu}$ is Cohen--Macaulay, and has a
		reduced special fiber equal to $\widetilde{A}_{\calG',\mu'}$. 
		If the admissible locus is irreducible, then $\widetilde{M}_{\underline{\calG}, \mu}$ has pseudo-rational singularities.
	\end{theorem}

	\begin{proof}
          As in the proof of \Cref{theorem_coherence_allp}, we reduce
          to the case $O=\breve O$, $G$ simple and note that
          $M_{\underline{\calG}, \mu}$ has a finite, birational,
          universally homeomorphic cover
          $M_{\underline{\calG_{\on{sc}}}, \mu}$ isomorphic to a
          subscheme of the Grassmannian
          $\Gr_{\underline{\calG_{\on{sc}}}}$ associated to the simply
          connected cover $G_{\on{sc}}\to G$.  In particular, by
          \Cref{thm_normality_classical_sch_vars}, its generic fiber
          is isomorphic to $S_{G,\mu}\cong\widetilde S_{G,\mu}$
          (Schubert varieties in characteristic $0$ are normal) and by
          Lemma \ref{l.inclusion} the special fiber contains
          $\widetilde{A}_{\calG',\mu'}$.
		
	Let $\calL_{\on{ad}}$ be the line bundle on $\Gr_{\underline{\calG_{\on{sc}}}}$ given by pullback of the determinant line bundle under the adjoint representation. 
			Its restriction to $M_{\underline{\calG_{\on{sc}}}, \mu}$ is ample, by \Cref{prop_ident_BD_Grass}. 
			By \eqref{equation.determinant.line}, we get an equality
			\begin{equation}
			\dim_k
                        H^0(\widetilde{A}_{\calG_{\on{sc}}',\mu'},
                        \calL_{\on{ad}}^{\otimes n})=\dim_{E}
                        H^0(S_{G_{\on{sc}}, \mu}, \calL_{\on{ad}}^{\otimes n}).
                      \end{equation}
                      This implies that
                    $\widetilde{M}_{\underline{\calG}, \mu}$ is normal
                    and its special fiber is reduced and equal to
                    $\widetilde{A}_{\calG',\mu'}$, compare with the
                    proof of \Cref{theorem_coherence_allp}.  The
                    Cohen--Macaulayness follows from flatness and that
                    of $\widetilde{A}_{\calG',\mu'}$ proven in
                    \Cref{theorem_coherence_allp}, see \cite[Lemma
                    5.7]{HR22}.  Moreover, if
                    $\widetilde{A}_{\calG',\mu'}=\widetilde{S}_{\calG',\mu'}$
                    is irreducible, then it has $F$-rational
                    singularities by
                    \Cref{theorem_frobenius_split_allp}, so
                    pseudo-rationality follows by
                    \Cref{lemma.pseudorational}.	\end{proof}
	
	\begin{remark}\label{rem: good primes equality mixed}
		Again, there is an equality $\widetilde{M}_{\underline{\calG}, \mu}= M_{\underline{\calG}, \mu}$ if and only if $\widetilde{A}_{\calG', \mu'}= A_{\calG', \mu'}$. 
		(Note that $\widetilde S_{G,\mu}=S_{G,\mu}$ because Schubert varieties in characteristic $0$ are normal.)
		This is ensured, for instance, when $p \nmid \lvert \pi_1(G)\rvert$, and may otherwise very well fail, see \cite[Corollary 9.2]{HLR24}.
	\end{remark}
 \begin{remark}
        	We note that, for $\mu$ minuscule, the
                $\calG_k$-scheme $\widetilde{A}_{\calG', \mu'}$ is
                related to the Witt vector affine Grassmannian of
                $\calG$, see \cite[Section 3]{AGLR22}.
              \end{remark}
              
	\begin{remark}
		If $\calG$ is special parahoric, then the admissible locus is irreducible, so $\widetilde{M}_{\underline{\calG}, \mu}$ has (pseudo-)rational singularities.
		For a complete list of triples $(G,\mu,\calG)$ with $G$ absolutely simple and $\mu$ minuscule such that the associated admissible locus is irreducible, the reader is referred to \cite[Theorem 7.1 (1)]{HPR20}.
	\end{remark}

		\begin{remark} \label{adjoint.groups.general.definition}
			The local models constructed in \cite{AGLR22} are invariant under passing to the adjoint group.
			So, if $G$ is not necessarily adjoint, we may define following \cite[Section 7.1]{HR20b} the local model as $\widetilde M_{\underline{\calG_{\on{ad}}},\mu_\ad}\otimes_{O_{E_\ad}}O_E$ where $\mu_\ad$ is induced by $\mu$ under $G\to G_\ad$ and $E_\ad\subset E$ denotes its reflex field. 
			Then \Cref{theorem_coherence_mixed_char} holds for this more general definition: this is clear if $E/E_\ad$ is unramified, and else follows from the method of proof.
		\end{remark}
We also get a complete description of the Picard group of the local model in mixed characteristic.

\begin{corollary}\label{cor_pic_local_model_mixed}
	Under \Cref{hyp_odd_unitary}, the following properties hold:
	\begin{enumerate}
		\item 
		\label{cor_pic_local_model_mixed.1}
		The restriction map
		$\mathrm{Pic}(\widetilde{M}_{\underline{\calG},\mu}) \to
		\mathrm{Pic}(\widetilde{A}_{\calG',\mu'})$ is an isomorphism.
		\item 
		\label{cor_pic_local_model_mixed.2}
		Let $G_i$ for $i=1,\dots, m$ be an enumeration of the simple factors of $G$ such that the image $\bar\mu_i$ of $\mu$ in the group $X_*(T_i)_I$ attached to $G_i$ is non-zero. 
		Then the restriction map 
		\begin{equation} \prod_{i=1}^{m}\mathrm{Pic}(\widetilde{\Fl}^{\tau_{i}}_{\calG'_{i}}) \to
		\mathrm{Pic}(\widetilde{A}_{\calG',\mu'}) \end{equation} is an isomorphism,
		where $\calG'_i$ is the associated parahoric $k\pot{t}$-model of $G'_i$ and the superscript $\tau_i$ indicates the connected component attached to $\mu'_i$.
		\item 
		\label{cor_pic_local_model_mixed.3}
		There is a commutative diagram:
		\begin{equation}
		\begin{tikzcd}[column sep=4.5cm, row sep=1cm,ampersand replacement=\&]
		\mathrm{Pic}(\widetilde{M}_{\underline{\calG},\mu}) \ar[r, ""] \ar[d, "\sim"]
		\& \mathrm{Pic}(\widetilde{S}_{G,\mu})
		\ar[d, "\sim"] \\
		\mathrm{Pic}(\widetilde{A}_{\calG',\mu'})  \& \prod_{i=1}^{m}\mathrm{Pic}(\widetilde{S}_{G_i,\mu_i}) \ar[d, "\prod_{i=1}^m \mathrm{deg}_i"]
		\\
		\prod_{i=1}^{m}\mathrm{Pic}(\widetilde{\Fl}^{\tau_{i}}_{\calG'_{i}}) \ar[u, "\sim"] \ar[r, "\prod_{i=1}^m c'_i"] \& \bbZ^m,
		\end{tikzcd}
		\end{equation}
		where the maps of Picard groups are induced by functoriality, $\mathrm{deg}_i$ denotes the degree homomorphism, and the $c'_i$ are the central charge homomorphisms for $\Fl_{\calG'_{i,\on{sc}}}$ translated to the other components.	
	\end{enumerate} 
\end{corollary}

\begin{proof}
  The proof is the same as in \Cref{cor_pic_local_model_equal}, and we briefly explain the necessary changes. 
  For \eqref{cor_pic_local_model_mixed.1}, we use \Cref{theorem_coherence_mixed_char} to know that $\widetilde{A}_{\calG',\mu'}$ equals the special fiber of $\widetilde{M}_{\underline{\calG},\mu}$. 
  The structure sheaf has vanishing higher cohomology by \Cref{cohomology.vanishing.lemma}, so line bundles lift uniquely. 
  
  Part \eqref{cor_pic_local_model_mixed.2} follows directly from  \Cref{cor_pic_local_model_equal} \eqref{cor_pic_local_model_equal.2}. 
  
 For \eqref{cor_pic_local_model_mixed.3}, we need to produce enough line bundles on the mixed characteristic local model $\widetilde{M}_{\underline{\calG},\mu}$, compare the proof of \Cref{line.bundle.extension}. 
 We have already seen how to construct the adjoint line bundle during \Cref{theorem_coherence_mixed_char}. 
 As for the kernel of the central charge, we define a map $\Gr_{\underline{\calG}} \to [\Spec \, O/\underline{\calG}_{t=0}]$ by reducing torsors to the subscheme defined by the principal ideal $t$, where $\underline{\calG}_{t=0}$ denotes the reduction of the $O\pot{t}$-group scheme $\underline{\calG}$ to $O$ via $t\mapsto 0$. 
 Pulling back line bundles of $[\Spec \, O/\underline{\calG}_{t=0}]$ to $\widetilde{M}_{\underline{\calG},\mu}$ yields the desired lifts of $\mathrm{ker}\,c$ with trivial generic fiber.
\end{proof}

	In the equicharacteristic case, we have seen in
	\Cref{theorem_coherence_allp} that local models have
	rational singularities. 
	Together with \Cref{theorem_coherence_mixed_char} at special level, this provides some motivation
	for the following:
	
	\begin{conjecture} \label{conjecture_mixed}
		The local model $\widetilde{M}_{\underline{\calG}, \mu}$ has pseudo-rational
		singularities.
	\end{conjecture}

	This would follow from \Cref{conjecture_pseudo_rational}.  For
        the purpose of proving \Cref{conjecture_mixed} for minuscule
        $\mu$, that is, the case relevant to Shimura varieties, it
        would suffice (by \Cref{theorem_coherence_mixed_char}) to also
        assume in \Cref{conjecture_pseudo_rational} that $R$ is
        Cohen--Macaulay and $R[\pi^{-1}]$ is regular (as in
        \cite[Proposition 2.13]{FW89}), and $F$-injective can be
        replaced by $F$-split.
	
	\subsection{Functoriality of local models}
	
	In this subsection, we discuss the behavior of our local models under certain maps of parahoric group schemes. 
	This is not used elsewhere in the paper, but plays an important role in \cite[Section 7]{AGLR22} for proving a comparison theorem between the power series approach of the present paper and the perfectoid approach in \cite{AGLR22}.
	
	In both equal and mixed characteristic, a morphism of pairs $(\calG,\mu)\to (\widetilde \calG,\widetilde \mu)$ is a map of $O$-group schemes $\calG\to \widetilde \calG$ which maps $\mu$ into $\widetilde \mu$ under the induced map of reductive $K$-groups $G\to \widetilde G$ in the generic fiber, compare \Cref{functoriality_remark}.
	In order to study functoriality properties, it is useful to
        base change the local model to the absolute integral closure
        $\bar O$ of $O$ with fraction field denoted $\bar K$.

	In equicharacteristic the formation of local models is functorial in the following sense:
	
	\begin{lemma}\label{lemma.functorliaty.equi}
	In equicharacteristic \textup{(}\Cref{equichar_LM_sec}\textup{)}, the association $(\calG,\mu)\mapsto \widetilde M_{\calG,\mu}\otimes_{O_E}{\bar O}$ from the category of pairs as above to the category of $\bar O$-schemes is functorial. 
	Under \Cref{hyp_odd_unitary}, it commutes with finite products, and the map $\calG\to \calG_\ad$ induces an isomorphism of $O_E$-schemes
	\begin{equation}\label{equation.adjoint.iso.equi}
	\widetilde M_{\calG,\mu}\cong \widetilde M_{\calG_\ad,\mu_\ad, O_E}, 
	\end{equation}
	where $\calG\to \calG_\ad$ is the map of parahoric $O$-models extending $G\to G_\ad$ and $\mu_\ad$ is the composite of $\mu$ with $G_{\bar K}\to G_{\ad, \bar K}$.
	\end{lemma}
	\begin{proof}
	This was proven in the course of \Cref{theorem_coherence_allp}, see especially the reduction in the beginning of its proof.
	Recall that for the isomorphism \eqref{equation.adjoint.iso.equi} and the commutation with finite products, the key fact is that $\widetilde M_{\calG,\mu}\otimes_{O_E}O_{\widetilde E}$ is normal for every finite field extension $\widetilde E\supset E$.
	\end{proof}

	\begin{remark}
	Using \Cref{remark.GL}, the special fiber of $\widetilde M_{\calG,\mu}\otimes_{O_E}{\bar O}$ is always reduced, so the base changed local model is normal and \Cref{lemma.functorliaty.equi} holds without assuming \Cref{hyp_odd_unitary}.	
	\end{remark}

	In mixed characteristic (\Cref{sec:mixed characteristic local models}), functoriality of $(\calG,\mu)\mapsto M_{\underline \calG, \mu}$ (or, its base change to $\bar O$) is subtle due to the auxiliary choices involved in the construction of the $O\pot{t}$-group lift $\underline \calG$.
	Here we point out two particularly interesting cases of functoriality: canonical $z$-extensions, making the connection to \cite[Section 2.6]{HPR20}, and embeddings into the Weil restriction of the split form, used in \cite[Section 7]{AGLR22}.

	\subsubsection{Canonical $z$-extensions following \cite[Section 2.4]{Lou23}}\label{z.extension.sec}

	Assume $K/\mathbb Q_p$ is of characteristic $0$ and use the notation introduced in \Cref{sec:mixed characteristic local models}. 
	In particular, $G$ satisfies \Cref{hyp_odd_unitary}, is adjoint, quasi-split and equipped with a quasi-pinning. 
	We lift the quasi-pinning along the simply connected cover $G_{\on{sc}}\to G$.
	This induces a map $\underline{\calG_{\on{sc}}}\to \underline{\calG}$ on the $O\pot{t}$-lifts by functoriality of extending birational group laws, compare \Cref{prop_lifts_bk}.	
			The maximal torus $T$ acts by inner automorphisms on $G_{\on{sc}}$, so we may form $\widetilde G:=G_{\on{sc}}\rtimes T$.
			By \cite[Lemme 2.4.2]{Lou23}, there is the $z$-extension 
			\begin{equation}\label{equation.canonical.extension}
			1\to T_{\on{sc}}\xrightarrow{t\mapsto (t,t^{-1})} \widetilde G\xrightarrow{(g,t)\mapsto gt} G\to 1
			\end{equation}
			with $\widetilde G_{\on{der}}=G_{\on{sc}}$ and $T\hookto \widetilde G, t\mapsto (1,t)$ being a maximal torus.
			By functoriality of extensions of birational group laws, the connected N\'eron model $\underline \calT$ acts on $\underline{\calG_{\on{sc}}}$ by inner automorphisms. 
			This allows us to define the $O\pot{t}$-group scheme $\underline{\widetilde\calG}:=\underline{\calG_{\on{sc}}}\rtimes \underline{\calT}$, which equals the model birationally glued from $(\underline{\calT_{\on{sc}}}\times \underline{\calT}, ({\underline{\calU_{a}}})_{a\in \Phi_G^{\on{nd}}})$ as in \Cref{prop_lifts_bk}.
			Moreover, it fits in a short exact sequence of $O\pot{t}$-group schemes
			\begin{equation}\label{equation.canonical.extension.lift}
			1\to \underline{\calT_{\on{sc}}}\to\underline{\widetilde\calG} \to \underline{\calG}\to 1,
			\end{equation}
			as can be seen by showing that $\underline{\calG}$ and the fppf quotient $\underline{\widetilde\calG}/\underline{\calT_{\on{sc}}}$ are solutions to the same birational group law, hence are isomorphic.
			The extension \eqref{equation.canonical.extension.lift} is called the \textit{canonical $z$-extension of $\underline{\calG}$}.

 The following lemma relates $\widetilde M_{\underline\calG,\mu}$ to the construction of local models via $z$-extensions as in \cite[Section 2.6]{HPR20}.
Here we view $\mu$ as a geometric cocharacter of $T$.

		\begin{lemma}\label{z.extensions.LM}
		Under \Cref{hyp_odd_unitary}, the map $\widetilde {\underline \calG}\to \underline\calG$ from \eqref{equation.canonical.extension} induces an isomorphism of $O_E$-schemes
		\begin{equation}
		M_{\widetilde {\underline \calG}, \widetilde\mu}\overset{\cong}{\longto} \widetilde M_{\underline\calG,\mu},
		\end{equation}
		where $\tilde \mu=(1,\mu)$ is viewed as a geometric cocharacter of $\widetilde G=G_{\on{sc}}\rtimes T$.
		\end{lemma}
		\begin{proof}
			Firstly, as $T$ is a maximal torus in both $G$ and $\widetilde G$, the cocharacters $\mu, \widetilde \mu$ have the same reflex field $E$.
			Thus, $\widetilde {\underline \calG}\to \underline\calG$ induces a finite birational universal homeomorphism on orbit closures 
			\begin{equation}\label{equation.map.orbit.closures}
			M_{\widetilde {\underline \calG}, \widetilde\mu}\to M_{\underline \calG,\mu},
			\end{equation} 
			which is an isomorphism on residue fields, see \cite[Corollary 2.3 and its proof]{HR22}.
			As $\widetilde{G}_{\on{der}}=G_{\on{sc}}$, the orbit closure $M_{\widetilde {\underline \calG}, \widetilde\mu}$ is normal by the proof of \Cref{theorem_coherence_mixed_char}. 
			So the map \eqref{equation.map.orbit.closures} induces $M_{\widetilde {\underline \calG}, \widetilde\mu}\cong \widetilde M_{\underline\calG,\mu}$ because the latter is normal by \Cref{theorem_coherence_mixed_char}.
		\end{proof}

	\subsubsection{Embedding into the Weil restriction of the split form}
	We record the following result concerning the functoriality of the construction $\calG\mapsto \underline{\calG}$, used in \cite{AGLR22}. 
	Recall the notation from \Cref{section_lifts} and consider the adjunction morphism
	\begin{equation}\label{eq_adjunction}
	G=\Res_{L/K}(G_0) \to \Res_{L/K}\Res_{\widetilde K/L}(H_0\otimes_\bbZ K)=\Res_{\widetilde K/K}(H_0\otimes_{\bbZ} K)=:\widetilde G,
	\end{equation}
	where $\widetilde{K}$ contains the Galois hull of $M/L$ and $H_0/\bbZ$ is the split form of $G_0$ induced by \eqref{equation.splitting.map}. 
	We assume the following:
	
	\begin{hypothesis}\label{not.the.triality}
		If $p=3$, then $G_0\otimes_K\breve K$ is not a triality form
                of type $D_4$.
              \end{hypothesis}

	Recall the $O\rpot{t}$-group lifts $\underline G$ from \Cref{lem_equiv_id_apart_bk}. 
	We equip $\widetilde G$ with the quasi-pinning induced from the pinning of $H_0\otimes_{\bbZ} K$, leading to the $O\rpot{t}$-group lift $\underline{\widetilde G}$.

	\begin{lemma}
		Under \Cref{not.the.triality}, the map \eqref{eq_adjunction} lifts to a locally closed immersion of $O\rpot{t}$-group schemes 
		\begin{equation}\label{equation.functoriality.map}
		\underline{G}\to \underline{\widetilde{G}},
		\end{equation}
		compatibly with reduction to $\kappa\rpot{t}$ for $\kappa=k,K$.
	\end{lemma}
	\begin{proof} 
		As the formation of $\underline G$ is compatible with restriction of scalars, we assume without loss of generality that $G=G_0$, so $L=K$.  
		\Cref{not.the.triality} ensures that the Galois hull of the fraction fields of the ring extension $O\pot{t} \to O_{M^{\rm{nr}}}\pot{v}$ is given by the fraction field of $O_{\widetilde{K}^{\rm{nr}}}\pot{v}$. 
		The map \eqref{equation.functoriality.map} exists by definition of $\underline{G}$ over the \'etale locus $U$ of $O\rpot{t} \to O_{M^{\rm{nr}}}\rpot{v}$, compare with \Cref{lemma.quasi.pinning.generic.extension}. 
		It can be further extended to $\Spec\, O\rpot{t}$ by taking the obvious inclusions for the models of the root groups \eqref{equation.root.models}, respectively the connected N\'eron models of tori, and by applying functoriality of solutions to birational group laws, compare with \cite[Proposition 3.3.9]{Lou23}. 
		This constructs \eqref{equation.functoriality.map}, which is a locally closed immersion by \cite[Proposition 2.2.10]{BT84}.
	\end{proof}
	
	Let $\widetilde S\subset \widetilde G$ be the maximal split subtorus contained in $\Res_{\widetilde K/K}(S)$.
	The inclusion of apartments
	\begin{equation}
	\scrA(G,S,K) \subset \scrA(\widetilde{G},\widetilde{S},K)
	\end{equation}
	is also compatible with the isomorphism \eqref{equation.identification.apartments.bk}.
	For a point $x\in \scrA(G,S,K)$, we denote its image by $\widetilde x \in \scrA(\widetilde{G},\widetilde{S},K)$.
	
	\begin{corollary}
		For $x\in \scrA(G,S,K)$, the map \eqref{equation.functoriality.map} extends to a locally closed immersion of the $O\pot{t}$-group schemes
		\begin{equation}\label{equation.parahoric.functoriality}
		\underline{\calG_{x}} \to \underline{\widetilde{\calG}_{\widetilde{x}}},
		\end{equation}
		constructed in \Cref{prop_lifts_bk}. 
		The map \eqref{equation.parahoric.functoriality} reduces to the canonical map of parahoric group schemes over $O$ and $\kappa\pot{t}$ with $\kappa=k,K$.
	\end{corollary}
	
	\begin{proof}
		Applying functoriality of solutions to birational group laws, it suffices to construct the maps between the models of roots groups and of tori, following \cite[Proposition 3.4.8]{Lou23}. 
		The resulting map is again a locally closed immersion by \cite[Proposition 2.2.10]{BT84}.
		That its reduction over $O$, respectively $\kappa\pot{t}$, is the expected map on parahoric group schemes is clear from the construction, compare \Cref{prop_lifts_bk}.
	\end{proof}
	
	Let us briefly return to the situation illustrated in \eqref{eq_adjunction} of the closed embedding of $G$ into the associated Weil-restricted split form $\widetilde{G}$. We denote by $\widetilde{\mu}$ the geometric conjugacy class of cocharacters of $\widetilde{G}$ obtained as the image of $\mu$. 
	The following compatibility at the level of local models plays a role in the proof of \cite[Theorem 7.23]{AGLR22}.

\begin{lemma}
	Under \Cref{hyp_odd_unitary} and \Cref{not.the.triality}, the map $\underline{\calG}:=\underline{\calG_{x}} \to \underline{\widetilde{\calG}_{\widetilde{x}}}:=\underline{\widetilde{\calG}}$ from \eqref{equation.parahoric.functoriality} induces a finite morphism
	\begin{equation}\label{equation_LM_functoriality}
	\widetilde{M}_{\underline{\calG},\mu} \to \widetilde{M}_{\underline{\widetilde{\calG}},\widetilde{\mu}}
	\end{equation}
	factoring uniquely through its scheme-theoretic image via a universal homeomorphism.
\end{lemma}

\begin{proof}
	By naturality of the Beilinson--Drinfeld Grassmannian, we obtain a map between the orbit closures, and hence the map \eqref{equation_LM_functoriality} by functoriality of seminormalizations \cite[Tag 0EUS]{StaProj}. 
	By projectivity of local models, it is enough to show that \eqref{equation_LM_functoriality} is injective on geometric points, which in turn can be tested on orbit closures.

In the generic fiber, the map $S_{G,\mu}\to S_{\widetilde G,\widetilde \mu}$ of Schubert varieties is a closed immersion because \eqref{eq_adjunction} is so.   
	In the reduced special fibers, the map is given by $\calA_{\calG',\mu'}\to \calA_{\widetilde\calG',\widetilde\mu'}$ on the respective admissible loci and is induced from $\Fl_{\calG'} \to \Fl_{\widetilde{\calG}'}$.

	It may happen that $\Fl_{\calG'} \to \Fl_{\widetilde{\calG}'}$ is not a monomorphism, because $\calG' \to \widetilde{\calG}'$ is a \textit{locally} closed immersion. 
	But this difference amounts to passing to a finite \'etale quotient of $\Fl_{\calG'}$ with isomorphic connected components (given by the affine flag variety of the flat closure of the immersion), which embeds into $\Fl_{\widetilde{\calG}'}$. 
	Since $\calA_{\calG',\mu'}$ is connected, this is enough to deduce injectivity of $\calA_{\calG',\mu'}\to \calA_{\widetilde\calG',\widetilde\mu'}$ on geometric points.
\end{proof}

	\bibliography{biblio.bib}
	\bibliographystyle{alpha}
	
\end{document}